\documentclass{article}
\usepackage[utf8]{inputenc}
\usepackage{amsmath,amsthm,amssymb,amsfonts}
\usepackage[a4paper, total={6in, 8in}]{geometry}
\usepackage{hyperref,color}
\usepackage{graphicx}
\usepackage[title]{appendix}
\usepackage[normalem]{ulem}

\newcommand{\bN}{ {\mathbb  N}}

\newcommand{\bC}{ {\mathbb  C}}

\newcommand{\ie}{{\it i.e.}}

\newtheorem{theorem}{Theorem}[section]

\newtheorem{lemma}[theorem]{Lemma}

\newtheorem{remark}[theorem]{Remark}

\title{Convergent expansions and bounds for the incomplete elliptic integral of the second kind near the logarithmic singularity}
\author{Dmitrii Karp \thanks{Holon Institute of Technology, Holon, Israel. Email: dimkrp@gmail.com}, Yi Zhang \thanks{Corresponding author. Department of Foundational  Mathematics, School of Science, Xi'an Jiaotong-Liverpool University, 
 Suzhou, 215123, China. The work of Y. Zhang was supported by XJTLU Research Development Fund No.\ RDF-20-01-12, the NSFC Young Scientist Fund No.\ 12101506 and the Natural Science Foundation of the Jiangsu Higher Education Institutions of China No.\ 21KJB110032. 
 Email: Yi.Zhang03@xjtlu.edu.cn}}
\date{\today}

\begin{document}

\maketitle

\begin{abstract}
We find two  series expansions for Legendre’s second incomplete elliptic integral $E(\lambda, k)$ in terms of recursively computed
elementary functions. Both expansions converge at every point of the unit square in the  $(\lambda, k)$ plane. Partial sums of the proposed expansions form a sequence of approximations to $E(\lambda,k)$  which are asymptotic when $\lambda$ and/or $k$ tend to unity, including when both approach the  logarithmic singularity $\lambda=k=1$ 
from any direction. Explicit two-sided  error bounds are given at each approximation order. These bounds yield a sequence of increasingly precise asymptotically correct two-sided inequalities for $E(\lambda, k)$. For the reader’s convenience we further present
explicit expressions for low-order approximations and numerical examples to illustrate their accuracy. Our derivations are based on
series rearrangements, hypergeometric summation algorithms and extensive use of the properties of the generalized  hypergeometric
functions including some recent inequalities. 
\end{abstract}

\bigskip


\bigskip


\section{Introduction} \label{SEC:intro}

Legendre's second elliptic integral (EI) is defined by \cite[(2.2)]{Carlson1961}
\begin{equation}\label{eq:E-defined}
E(\lambda,k)=\int_{0}^{\lambda}\frac{\sqrt{1-k^2t^2}}{\sqrt{1-t^2}}dt.
\end{equation}
It can be expressed in terms of Appell's hypergeometric function \cite[(2.7)]{Carlson1961}
$$
F_1(\alpha;\beta,\beta';\gamma;x,y)=\sum\limits_{m,n=0}^{\infty}\frac{(\alpha)_{m+n}(\beta)_{m}(\beta')_{n}}{(\gamma)_{m+n}}\frac{x^my^n}{m!n!}
$$
as follows  \cite[(2.9)]{Carlson1961}
\begin{equation}\label{eq:E-Appell}
E(\lambda,k)={\lambda}F_1(1/2;1/2,-1/2;3/2;\lambda^2,k^2\lambda^2).
\end{equation}
The double series defining $F_1$ converges in the domain $|\lambda^2|<1$, $|k^2\lambda^2|<1$ in the space $\mathbb{C}^2$ of the complex variables $(\lambda,k)$ and defines an analytic function there. Clearly, the bi-disk $|\lambda|<1$, $|k|<1$ is (properly) contained in the convergence domain. The function $F_1$ in \eqref{eq:E-Appell} can be analytically continued to the domain 
$$
\left\{|k|<1, |k\lambda|>1, |\arg(-\lambda^2)|<\pi, |\arg(-k^2\lambda^2)|<\pi  \right\}
$$
according to \cite[Proposition~5]{Bezrodnykh2017}  and further to $|k|>1$, $|\lambda|>1$ and the same restrictions on the arguments via the reflection relation \cite[(19.7.4)]{NIST}
$$
kE(\lambda,1/k)=E(\lambda/k,k) - (1-k^2)F(\lambda/k,k),
$$
where 
$$
F(\lambda,k)=\int_0^{\lambda}\frac{dt}{\sqrt{(1-t^2)(1-k^2t^2)}}={\lambda}F_1(1/2;1/2,1/2;3/2;\lambda^2,k^2\lambda^2)
$$
is the first Legendre's incomplete elliptic integral \cite[(2.8)]{Carlson1961}.
Note that $E(\lambda,1)=\lambda$ for each $0<\lambda<1$, while $E(1,k)=E(k)$, which is  the complete elliptic integral of the second kind.  Expansions for $F(\lambda,k)$ analogous to those derived in this paper for $E(\lambda,k)$ were found by the first author jointly with S.M.\:Sitnik in \cite{Dima2007}.

The following two symmetric standard EIs are defined in~\cite{Carlson1961, Carlson1977, CarlsonGustaf1985, Carlson1994} as follows:
\begin{align*}
R_F(x, y, z) & = \frac{1}{2} \int_0^\infty \frac{dt}{\sqrt{(t + x) (t + y) (t + z)}}, \\
R_D(x, y, z) & = \frac{3}{2} \int_0^\infty \frac{dt}{(t + z) \sqrt{(t + x) (t + y) (t + z)}},
\end{align*}
and related to $E(\lambda, k)$ by \cite[(4.2)]{Carlson1979}
\begin{equation} \label{EQ:symmetrictoLegendre}
E(\lambda, k) = \lambda R_F(1 - \lambda^2, 1 - k^2 \lambda^2, 1) - \frac{1}{3} k^2 \lambda^3 R_D(1 - \lambda^2, 1 - k^2 \lambda^2, 1).  
\end{equation}
Asymptotic expansions for $E(\lambda, k)$ near the point $(1, 1)$ appeared in~\cite{Vel1969,Kaplan1948}. For symmetric elliptic integrals with one of the parameters going to infinity, the first (and the second in some cases) term of the asymptotic expansion of $R_F$, $R_D$, and $R_J$, as well as a quite accurate bounds for the remainder, have been obtained by Carlson and Gustafson \cite{Carlson1994}.  Moreover, for all the symmetric EIs, they also considered the case of several parameters going to infinity.  The first approximation of Carlson and Gustafson has been extended to the general zero-balanced Appell function $F_1$ by the first author in \cite{Dima2008}. Complete convergent expansions for symmetric EIs (and not only first terms) have been obtained earlier by Carlson using Mellin transform techniques \cite{Carlson1985}, but computation of the higher order terms is not at all  straightforward and the error bounds are not satisfactory \cite[Section~3]{Carlson1985}.  The complete asymptotic expansions with recursively computed terms and explicit error bounds at each approximation order were obtained by L\'{o}pez in \cite{Lopez2000, Lopez2001} for various asymptotic regimes.  Formula \eqref{EQ:symmetrictoLegendre} allows converting his results into the  asymptotic approximations for Legendre's EI  $E(\lambda, k)$ as $\lambda\to1$ (while $k$ is fixed or tends to $1$ as well).  Details of these conversion are given in the Appendix to this paper. The resulting approximation takes the form
\begin{equation} \label{EQ:CGexpansion1}
E(\lambda, k) = \lambda (1 - k^2 \lambda^2) \ln \frac{4}{\sqrt{1 - \lambda^2} + \sqrt{1 - k^2 \lambda^2}} + k^2 \lambda^3 + r_1,
\end{equation}
with the remainder $r_1$ satisfying \eqref{EQ:Lbound1} and \eqref{EQ:CGbound1}.  For a more comprehensive overview of the theories, algorithms, and applications of elliptic integrals, we refer to~\cite{Carlson2010, DLMF}.

In this paper we will derive two types of convergent  series directly for $E(\lambda,k)$ which are also asymptotic when either $\lambda$ or $k$ or both tend to $1$.  Each of the two series converges at each point of the unit square in the $(\lambda,k)$ plane.  Convergence is uniform on compact subsets of the closed unit square with one boundary segment removed ($\lambda=1$ ($k=1$) for the first (second) expansion).   We further furnish explicit two-sided error bounds at each approximation order.  Hence, our results can also be interpreted as a sequence of asymptotically precise (as $\lambda,k\to1$) two-sided inequalities for the second incomplete EI $E(\lambda,k)$.  Our derivation does not rely on asymptotic methods and uses standard analytic techniques combined with the algorithms of symbolic computation and some recent and rather accurate inequalities for the generalized hypergeometric function.   This leads to high-precision approximations which are much better than those present in the literature so far.  We demonstrate this numerically in the ultimate section of the paper. For example, our first order approximation is given by 
$$
E_1(\lambda, k)= (\lambda-1/\lambda)\sqrt{1 + (\lambda^2 (1 - k^2))/(1 - \lambda^2)} - \frac{1-k^2}{4}\ln\frac{1-\lambda}{1+\lambda} +1/\lambda.
$$
This approximation is also an upper bound.  We further propose a sequence of more precise refined approximations which do not constitute (neither upper nor lower) bounds. For instance, the first order refined approximation of the first kind is given  by
\begin{multline*}
\hat{E}_1(\lambda, k)= (\lambda-1/\lambda)\sqrt{1 + \frac{\lambda^2(1 - k^2)}{1 - \lambda^2}} - \frac{(101+19k^2)(1-k^2)}{32(7+8k^2)}\ln\frac{1-\lambda}{1+\lambda} +1/\lambda
\\
-\frac{675\sqrt{2}(1-k^2)^{3/2}}{128(7+8k^2)\sqrt{15-7\lambda^2-8\lambda^2k^2}} \ln\frac{\sqrt{15-7\lambda^2-8\lambda^2k^2}+\lambda\sqrt{8(1-k^2)}}{\sqrt{15-7\lambda^2-8\lambda^2k^2}-\lambda\sqrt{8(1-k^2)}}.
\end{multline*}
Table~\ref{TABLE:1} in Section~5 shows a remarkable accuracy of this approximation. 

The paper is organized as follows. In Section~2 succeeding this introduction we rederive two known series expansions for $E(\lambda,k)$ using partial fractions and the generating function for Legendre's polynomials and find new bounds for the remainders.  These expansions then serve as the starting points for new expansions established in Sections 3 and 4.  Both of them converge for any fixed $(\lambda,k)\in(0,1)\times(0,1)$. The partial sums of the first expansion derived in Section~3  form an asymptotic series as $k\to1$ which is uniform with respect to $\lambda$ lying in any subset of the unit square with bounded ratio $(1-k)/(1-\lambda)$.  In a similar fashion, the partial sums of the second expansion derived in Section~4 form an asymptotic series as $\lambda\to1$ which is uniform with respect to $k$ lying in any subset of the unit square with bounded ratio $(1-\lambda)/(1-k)$. In Section~5 we present the results of numerical experiments illustrating high precision of our first and second approximations and even more so for the refined approximations obtained by incorporating the error bounds into the formulas. Finally, we included a short appendix containing a conversion of the first approximations for the symmetric elliptic integrals due to Carlson-Gustafson \cite{Carlson1994} and L\'{o}pez \cite{Lopez2000}  and their error bounds into the corresponding results for the incomplete Legendre's second elliptic integral.


\section{Expansions of Byrd-Friedman and Carlson revisited} \label{SEC:carlson}

In this section  we present two auxiliary series expansions which will be the cornerstones for the main results given in Sections~\ref{SEC:first} 
and~\ref{SEC:sec}. The first expansion can be regarded as an equivalent form of a known expansion due to Byrd-Friedman~\cite{ByrdFriedman1971}, while the second one is 
derived from an expansion by Bille C.\ Carlson~\cite{Carlson1961} via certain hypergeometric transformations. The error bounds given in this section appear to be new. 

To deduce the first expansion, we need the following lemma. 

\begin{lemma} \label{LEM:first}
For an integer $j \geq 1$ and $0\le\lambda<1$, we have the following identity:
\begin{multline} 
\int_0^{\lambda} \frac{t^{2 j} dt}{(1 - t^2)^{j}}= 
\frac{\lambda^{2j+1}}{2j+1}{}_2F_{1}(j,j+1/2;j+3/2;\lambda^2)
\\
=\frac{\lambda^{2 j + 1}}{(1 - \lambda^2)^j} 
+ (-1)^{j} \frac{(1/2)_j}{(j - 1)!} \ln\frac{1 - \lambda}{1 + \lambda}
+  \frac{1}{\lambda} \sum_{n = 0}^{j - 1} (-1)^{n - 1} \frac{(1/2 - j)_n}{(1 - j)_n} \left(\frac{\lambda^2}{1 - \lambda^2} \right)^{j - n}, \label{EQ:first}
\end{multline}
where 
$$
(b)_0 = 1, \quad (b)_n = b (b + 1) (b + 2) \cdots (b + n - 1), \, n\ge1, 
$$
is the Pochhammer symbol (or the rising factorial).
\end{lemma}

\begin{proof}
The first equality is the direct consequence of Euler's integral representation \cite[15.6.1]{NIST}. To establish the second equality by using integration by parts, we have
\begin{align*}
I & = \int_0^{\lambda} \frac{t^{2 j} dt}{(1 - t^2)^{j}} = \frac{1}{2 j + 1} \int_0^{\lambda} \frac{d(t^{2 j + 1})}{(1 - t^2)^{j}} 
\\  & = \frac{1}{2 j + 1} \left[ \frac{t^{2 j + 1}}{(1 - t^2)^{j}}\bigg|_0^{\lambda} 
- \int_0^{\lambda} t^{2 j + 1} d\left( \frac{1}{(1 - t^2)^j} \right) \right] \\
  & = \frac{1}{2 j + 1} \cdot \frac{\lambda^{2 j + 1}}{(1 - \lambda^2)^j} + \frac{2 j}{2 j + 1} I - \frac{2 j}{2j + 1} \int_0^{\lambda} \frac{t^{2 j} dt}{(1 - t^2)^{j+ 1}},
\end{align*}
where we used 
$$
\frac{t^{2j+2}}{(1-t^2)^{j+1}}=\frac{t^{2j}}{(1-t^2)^{j+1}}-\frac{t^{2j}}{(1-t^2)^{j}}
$$
in the last equality. Thus, we get
\[
I = \frac{\lambda^{2j+1}}{(1-\lambda^2)^j}-2j\int_0^{\lambda}\frac{t^{2j}dt}{(1-t^2)^{j+ 1}}. 
\]
Substituting the closed formula from \cite[Lemma 1]{Dima2007} for the integral on the right-hand side of the above identity, we arrive at~\eqref{EQ:first}.
\end{proof}

For conciseness of the subsequent formulas it is convenient to introduce the parameter 
\begin{equation}\label{eq:beta-defined}
\beta=\beta(\lambda,k)=\frac{1-\lambda^2}{1-k^2}.
\end{equation}

\begin{theorem} \label{THM:first}
Suppose 
\begin{equation} \label{EQ:cond1}
\beta>\lambda^2~\Leftrightarrow~1-\lambda^2k^2<2(1-\lambda^2). 
\end{equation}
For each integer $N\ge1$, we have the following decomposition
\begin{align}
E(\lambda, k) & = \lambda \sum_{j=0}^{N}(-1)^j\frac{(-1/2)_j}{j!} \left[\frac{\lambda^2}{\beta}\right]^j 
+  \ln\left(\frac{1 - \lambda}{1 + \lambda}\right) \sum_{j = 1}^N \frac{(-1/2)_j (1/2)_j}{j! (j - 1)!} (1 - k^2)^j  \nonumber \\
& \quad + \frac{1}{\lambda} \sum_{j = 1}^N\left[\frac{\lambda^2}{\beta}\right]^{j}  \sum_{n = 0}^{j -1} (-1)^{j + n - 1} \frac{(-1/2)_j (1/2 - j)_n}{j! (1 - j)_n}  \left( \frac{1 - \lambda^2}{\lambda^2}\right)^{n} 
+ R_{1, N}(\lambda, k). \label{EQ:expansion1}
\end{align}
The remainder $R_{1, N}(\lambda, k)$ satisfies the inequality 
\begin{equation} \label{EQ:bound1}
|R_{1, N}(\lambda, k) | \leq \frac{\lambda (1 - \lambda^2) (2 N - 1)!!}{N 2^{N+ 2} (N+ 1)!}  \left[\frac{\lambda^2}{\beta}\right]^{N + 1}.
\end{equation}
\end{theorem}

\begin{remark} \label{REM:first}
It is apparent from the error bound~\eqref{EQ:bound1} that the expansion~\eqref{EQ:expansion1} is convergent for any fixed $\lambda$ and $k$ satisfying~\eqref{EQ:cond1} and asymptotic when $[(1 - k) \lambda]/(1 - \lambda) \rightarrow 0$ . 
\end{remark}

\begin{remark} \label{REM:complex1}
Expansion~\eqref{EQ:expansion1} also converges for complex $\lambda$ and $k$ 
satisfying $|\lambda^2/\beta| < 1$.
\end{remark}

\begin{proof}
Set $k'^{2} = 1 - k^2$. Expanding $\left[ 1 + (k'^{2} t^2)/(1 - t^2)\right]^{1/2}$ into the binomial series and integrating term-wise, we have
\begin{align*}
E(\lambda, k) & = \int_{0}^\lambda\sqrt{\frac{1-k^2t^2}{1-t^2}}dt=\int_{0}^\lambda dt \left( 1 + \frac{k'^{2} t^2}{1 - t^2}\right)^{1/2} \\
& = \int_0^\lambda dt \left( \sum_{j = 0}^\infty (-1)^j \frac{(-1/2)_j}{j!} \frac{k'^{2 j} t^{2 j}}{(1 - t^2)^j} \right) \\
& = \sum_{j = 0}^N (-1)^j \frac{(-1/2)_j}{j!} k'^{2 j} \int_0^{\lambda} \frac{t^{2 j} dt}{(1 - t^2)^{j}} 
 + \sum_{j =N+1}^\infty (-1)^j \frac{(-1/2)_j}{j!} k'^{2 j} \int_0^{\lambda} \frac{t^{2 j} dt}{(1 - t^2)^{j}}.
\end{align*}
Writing the integral in the first sum as~\eqref{EQ:first}, we get~\eqref{EQ:expansion1} with the remainder given by
\begin{multline*}
R_{1, N}(\lambda,k) = \sum_{j=N+1}^\infty (-1)^j\frac{(-1/2)_j}{j!} k'^{2 j} \int_0^{\lambda} \frac{t^{2 j} dt}{(1 - t^2)^{j}}
\\
=(-1)^{N+1}\sum_{j=0}^\infty (-1)^j\frac{(-1/2)_{N+j+1}}{(N+j+1)!} k'^{2(N+j+1)} \int_0^{\lambda}
\frac{t^{2(N+j+1)} dt}{(1-t^2)^{N+j+1}}=(-1)^{N+1}\underbrace{(-a_0+a_1-a_2+\cdots)}_{=S},
\end{multline*}
where $a_j>0$ for all $j$, so that it is clearly an alternating series. The following argument shows that each term is
smaller in absolute value than the previous one:
\begin{multline*}
\frac{(2 j - 1)!!}{2^{j + 1} (j + 1)!} (1 - k^2)^{j + 1} \int_0^{\lambda} \frac{t^{2 j + 2} dt}{(1 - t^2)^{j+1}} = 
\frac{2 j - 1}{2 ( j + 1)} \frac{(2 j - 3)!!}{2^j j!} (1 - k^2)^j \int_0^{\lambda} \frac{t^{2 j}}{(1 - t^2)^{j}} \frac{(1 - k^2) t^2}{1- t^2}dt \\ \leq \frac{(2 j - 3)!!}{2^j j!} (1 - k^2)^j \int_0^{\lambda} \frac{t^{2 j}}{(1 - t^2)^{j}} \frac{(1 - k^2) \lambda^2 }{1- \lambda^2}dt \leq \frac{(2 j - 3)!!}{2^j j!} (1 - k^2)^j \int_0^{\lambda} \frac{t^{2 j}}{(1 - t^2)^{j}} dt.
\end{multline*}
The last inequality follows from the condition~\eqref{EQ:cond1}. 
Thus, by the alternating series test, we see that the absolute value of the remainder $|R_{1,N}(\lambda,k)|$ 
is bounded by
\[
a_0=\frac{(2 N - 1)!!}{2^{N+1} (N + 1)!} (1 - k^2)^{N+1} \int_{0}^\lambda \frac{t^{2 N + 2} dt}{(1 - t^2)^{N+1}}.
\]
Indeed, $|R_{1,N}(\lambda,k)|=|S|=-S$ and  $-a_0\le S \le 0$.
Next, we prove the following asymptotically exact (as $\lambda \rightarrow 1$) inequality
\begin{equation} \label{EQ:inequal1}
f_1(\lambda) := \int_0^\lambda \frac{t^{2 b} dt}{(1 - t^2)^b} \leq f_2(\lambda) := \frac{\lambda^{2 b + 1}}{2 (b-1)(1 - \lambda^2)^{b - 1}},
\end{equation}
which is valid for each $\lambda \in (0, 1)$ and $b > 1$. Indeed, $f_1(0) = f_2(0) = 0$ and 
\[
\frac{f_1^{'}(\lambda)}{f_2^{'}(\lambda)} = \frac{2 (b - 1)}{1 + 2 b - 3 \lambda^2} < 1, \quad \lambda \in (0, 1).
\]
It remains to note that inequality~\eqref{EQ:inequal1} immediately implies~\eqref{EQ:bound1}.
\end{proof}

\begin{remark}\label{REM:sec}
In~\cite[page 301, 903.01]{ByrdFriedman1971}, Byrd and Friedman presented the following expansion
\begin{equation} \label{EQ:ByrdFriedman}
E(\phi, k) = \sum_{m = 0}^\infty {1/2 \choose m} k'^{2 m} d_{2 m}(\phi), 
\end{equation}
where $\lambda = \sin(\phi)$, and $d_{2 m}(\phi)$'s are given by a linear recurrence relation and initial values. Since
\[
d_{2 m}(\phi) = \int_0^{\sin(\phi)} \frac{t^{2 m} dt}{(1 - t^2)^{m}},
\]
we see that~\eqref{EQ:expansion1} is an equivalent form of~\eqref{EQ:ByrdFriedman}.
\end{remark}

To derive the second expansion, we will need the following lemma from \cite[Lemma~2]{Dima2007}. The symbol ${}_2F_1$ represents the Gauss hypergeometric function and $P_n$ is Legendre's polynomial \cite[section~14.7(i)]{NIST}.

\begin{lemma} \label{LEM:sec}
\begin{enumerate}
\item [(i)] The function $F_n(x) := {_2}F_1(-n, 1/2; 1; x)$ is expressed in terms of Legendre's  polynomials as:
\begin{equation} \label{EQ:Lengendre}
F_n(x) = (1-x)^{n/2} P_n\left( \frac{2 - x}{2 \sqrt{1 - x}}\right).
\end{equation}
\item [(ii)] For each $n \geq 0$, the function $F_n(x)$ is decreasing on  $[0, 1]$, so that
\[
F_n(1) = \frac{(1/2)_n}{n!} \leq F_n(x) \leq F_n(0) = 1.
\] 
\item [(iii)]  For $x \in [1, 2]$,  the function $F_n(x)$ is monotone decreasing when $n$ is an odd and satisfies the following bounds
\[
F_n(2) = 0 \leq F_n(x) \leq F_n(1) = \frac{(1/2)_n}{n!}  \leq 1.
\]
If $n$ is an even, then the function $F_n(x)$ has a single minimum at $x_{\min} \in (1, 2)$, and satisfies  the following bounds
\[
0 < F_n(x) \leq F_n(2) = \frac{n!}{2^n (n/2)!^2} \leq 1.
\]
\item [(iv)] For $x > 2$, the function $F_n(x)$ has the sign $(-1)^n$ and increases (decreases) for even (odd) $n$, so that
\begin{equation}\label{EQ:inequal2F1}
|F_n(x)| \leq (x - 1)^n.
\end{equation}
\end{enumerate}
\end{lemma}

We are now ready to present our second auxiliary expansion. 

\begin{theorem} \label{THM:sec}
Suppose
\begin{equation} \label{EQ:cond2}
\beta{k^2}<1, 
\end{equation}
where $\beta$ is defined in \eqref{eq:beta-defined}. For each integer $N\ge1$, we have the following decomposition\emph{:}
\begin{multline}\label{EQ:expansion2}
E(\lambda, k) = E(k) 
-\sqrt{(1 - \lambda^2) (1 - k^2)}  \cdot
\sum_{m = 0}^{N - 1} \left(\frac{1}{2m+1}+ \frac{\beta{k^2}}{2m+3}\right)(1-\lambda^2)^m  \\
\cdot {_2}F_1(-m, 1/2; 1; (1 - k^2)^{-1}) + R_{2,N}(\lambda,k)
\\
= E(k)-(1-\lambda^2)\cdot
\sum_{m = 0}^{N-1}(-1)^m\left(\frac{k}{2m+1}+ \frac{\beta{k^3}}{2m+3}\right)\left[\frac{k^2(1-\lambda^2)}{1-k^2}\right]^{m-1/2}  \\
\cdot {_2}F_1(-m, 1/2; 1; 1/k^2) + R_{2,N}(\lambda,k)
,
\end{multline}
where $E(k)=E(1, k)$ is the complete EI of the second kind. The bound for the remainder is given by
\begin{equation} \label{EQ:bound2}
|R_{2, N}(\lambda, k) | \leq \frac{(N+1)\left(\beta{k^2}\right)^N\sqrt{(1 - \lambda^2)(1 - k^2)}}{(N+1/2)(N+3/2)(1-\beta{k^2})}
\end{equation}
for $1/2 \leq k^2 < 1$, and 
\begin{equation} \label{EQ:bound3}
|R_{2, N}(\lambda, k) | \leq \frac{(N + 1)(1-\lambda^2)^N}{\lambda^{2}(N + 1/2)(N + 3/2)}  \sqrt{(1 - \lambda^2) (1 - k^2)}
\end{equation}
for $0 < k^2 \leq 1/2$.
\end{theorem}

\begin{remark} \label{REM:third}
It is clear from the error bounds~\eqref{EQ:bound2} and~\eqref{EQ:bound3} that the expansion~\eqref{EQ:expansion2} is convergent for any fixed $\lambda$ and $k$ satisfying~\eqref{EQ:cond2} and is asymptotic as $(1 - \lambda)/(1 - k) \to 0$. 
\end{remark}

\begin{remark}
The set of points satisfying either condition~\eqref{EQ:cond1} or  condition~\eqref{EQ:cond2}  covers the entire unit $(k,\lambda)$ square \emph{(}see Figure~\ref{fig:r3}\emph{)}. 
\end{remark}

\begin{remark} \label{REM:complex2}
Expansion~\eqref{EQ:expansion2} also holds for complex $\lambda$ and $k$ 
satisfying $|((1 - \lambda^2) k^2)/(1 - k^2)| < 1$.
\end{remark}

\begin{figure}[htb]
\begin{center}
\includegraphics[width=10cm]{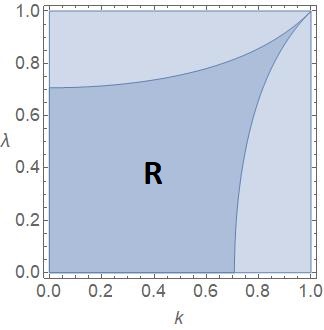}
\end{center}
\caption{The set of points in the unit square satisfying~\eqref{EQ:cond1} or~\eqref{EQ:cond2} covers the unit square. The deep blue region $R$ comprises the points satisfying both~\eqref{EQ:cond1} and~\eqref{EQ:cond2}.}
\label{fig:r3}
\end{figure}

\begin{proof}
By~\cite[(3.1), (3.2)]{Carlson1961}, we have the expansion around $\lambda = k = 0$:
\[
E(\lambda, k) = \sum_{m = 0}^\infty k^m P_n\left(\frac{k + k^{-1}}{2} \right)\lambda^{2 m + 1}  \left[\frac{1}{2 m + 1} - \frac{k^2 \lambda^2}{2 m + 3} \right].
\]
Applying item~(i) of Lemma~\ref{LEM:sec} to the above identity, we get
\begin{equation} \label{EQ:Carlsonformula}
E(\lambda, k) = \sum_{m = 0}^\infty \lambda^{2 m + 1} {_2}F_1(-m, 1/2; 1; 1 - k^2) \left[\frac{1}{2 m + 1} - \frac{k^2 \lambda^2}{2 m + 3} \right],
\end{equation}
which is valid for $\lambda \in (0, 1)$ and $|k| < 1/\lambda$. Next, we will employ the reflection-type relation
\begin{equation}\label{EQ:reflection}
E(\lambda, k) = E(k) - \sqrt{1 - k^2} \cdot E\left(\sqrt{1 - \lambda^2}, \sqrt{-k^2/(1 - k^2)} \right), 
\end{equation}
which can be verified by representing the second incomplete EI as the difference of the second complete EI $E(k)$ and 
the integral over the interval $(\lambda, 1)$ and then introducing the integration variable $v^2 = 1 - t^2$.  
Substituting~\eqref{EQ:Carlsonformula} into the second term on the right side of~\eqref{EQ:reflection} and splitting the 
corresponding series, we derive~\eqref{EQ:expansion2}, which is valid for $\lambda \in (0, 1)$ and $((1 - \lambda^2) k^2)/(1 - k^2) < 1$, with the remainder given by
\begin{multline*}
R_{2, N}(\lambda, k) = -\sqrt{(1-\lambda^2)(1-k^2)}\sum_{m = N}^{\infty}
\left(\frac{1}{2m+1}+\frac{\beta{k^2}}{2m+3}\right) (1-\lambda^2)^m \\
\cdot {_2}F_1(-m, 1/2; 1; (1 - k^2)^{-1}).
\end{multline*}
Using items~(ii),~(iii),~(iv) of Lemma~\ref{LEM:sec} and condition~\eqref{EQ:cond2}, we obtain
\[
|R_{2,N}(\lambda, k) | \leq \sqrt{(1 - \lambda^2) (1 - k^2)} \cdot 
\sum_{m = N}^{\infty} \frac{4(m + 1)}{(2m+1)(2m+3)} \left(\beta{k^2}\right)^m 
\]
for $1/2 \leq k^2 < 1$, and
\[
|R_{2,N}(\lambda,k)| \leq \sqrt{(1-\lambda^2)(1-k^2)}\sum_{m = N}^{\infty} \frac{4(m+1)}{(2m+1)(2m+3)}(1-\lambda^2)^m 
\]
for $0 < k^2 \leq 1/2$. Applying the following inequality
\begin{multline*}
\sum_{m = N}^\infty \frac{4(m+1)x^m }{(2m+1)(2m+3)} = x^N\sum_{s=0}^\infty\frac{4(N+s+1)x^s}{(2N+2s+1) (2N + 2 s +3)}  \\ \leq \frac{4x^N(N+1)}{(2N+1)(2N+3)}\sum_{s = 0}^{\infty}x^s =\frac{4x^N (N+1)}{(2N+1)(2N+3)(1-x)},
\end{multline*}
which is valid for $x \in (0, 1)$, we get~\eqref{EQ:bound2} and~\eqref{EQ:bound3}. Finally, the second equality in \eqref{EQ:expansion2} follows on application of Pfaff's transformation \cite[15.8.1]{NIST}. 
\end{proof}

\section{The first asymptotic expansion} \label{SEC:first}

For each $n \geq 0$, set 
\begin{equation} \label{EQ:hypergeo}
s_n(x) = \sum_{j = n + 1}^\infty \frac{(-1/2)_j (1/2 - j)_n}{j! (1 - j)_n} (-x)^j.
\end{equation}
Making a change of the summation variable $m = j - n - 1$, we get
\begin{equation}\label{eq:sn4F3}
s_n(x)=\frac{(-1/2)_{n+1}(3/2)_n}{(n+1)!n!}(-x)^{n+1} {}_4F_3\left(\begin{matrix}1,&1,&1/2 + n&3/2 + n \\&3/2,&1+ n,&2 + n \end{matrix} \Bigg| -x \right),
\end{equation}
where ${}_4F_3$ represents the generalized hypergeometric function. Note that this formula implies that $s_n(x)$ is holomorphic in the cut $x$-plane $\mathbb{C}\setminus(-\infty,-1]$. In particular, it is holomorphic in the unit disk $|x|<1$ with a branch point at $x=-1$.

Next, we present a linear recurrence relation for $s_n(x)$ obtained by the method of creative telescoping~\cite{Zeilberger1991} and the initial values in terms of elementary functions.

\begin{lemma} \label{LEM:ct}
The function $s_n(x)$ satisfies the following third-order linear inhomogenous recurrence relation
\begin{equation} \label{EQ:rec}
4 (n + 2) (n + 3) s_{n + 3}(x) = a_n(x) s_{n + 2}(x) + b_n(x) s_{n + 1}(x) + c_n(x) s_n(x) + d_n(x),
\end{equation}
where
\begin{align*}
a_n(x) & = - (2n + 3) (2nx + 5x - 4n -8), \\
b_n(x) & = (2n + 3) (4nx + 4x - 2n - 1), \\
c_n(x) & = -4n (1 + n) x, \\
d_n(x) & = - a_n(x) g(n+2, n+3) - b_n(x) [g(n+1, n+2) + g(n + 1, n+3)] \\
           & \quad - c_n(x) [g(n, n+1) + g(n,n +2) + g(n, n+3)] - h(n, n +4)
\end{align*}
with
\begin{align*}
g(n, j) & = \frac{(-1/2)_j (1/2 - j)_n}{j! (1 - j)_n} (-x)^j, \\ 
h(n, n+ 4) & = -\frac{7}{4} \frac{(-1/2)_{n+4} (-7/2 -n)_n}{(n + 2)! (-3 -n)_n} (-x)^{n + 4} 
\end{align*}
for each $j, n \geq 0$. The initial values are given by
\begin{align}
s_0(x) & = \sqrt{1 + x} - 1,  \label{EQ:s0} \\ 
s_1(x) & = \frac{1}{4}\left[-2 + 2\sqrt{1 + x} + x \left(2\ln2-1-2\ln(1 + \sqrt{1 + x})\right)\right], \label{EQ:s1} \\ 
s_2(x) & = -\frac{3 x^2}{16 (1 + \sqrt{1 + x})^2 (-1 + \sqrt{1 + x})} \left[-\frac{1}{2} (x - \frac{8}{3}) (1 + \sqrt{1 + x})  \ln (1 + \sqrt{1 + x}) \right. \nonumber\\
           &  \quad  + \frac{1}{2} x (1 + \sqrt{1 + x})  \ln (-1 + \sqrt{1 + x}) \nonumber \\
           & \quad + \left((x - \frac{4}{3}) \ln 2 - \frac{1}{2} x \ln x - \frac{13}{12} x + 1 \right) \sqrt{1 + x} + (x - \frac{4}{3}) \ln 2 - \frac{1}{2} x \ln x 
      \left. - \frac{x}{12} - 1\right].\label{EQ:s2} 
\end{align}
\end{lemma}
\begin{proof}
Using Koutschan's {\tt Mathematica} package {\tt HolonomicFunctions.m}~\cite{Christoph2010} that
implements Chyzak's algorithm~\cite{Chyzak2000}, we derive the following  linear  recurrence relation for 
the generic term $g(n, j)$ in~\eqref{EQ:hypergeo}: 
\begin{equation} \label{EQ:ct}
4 (n + 2) (n + 3) g(n + 3, j) = a_n(x) g(n + 2, j) + b_n(x) g(n + 1, j) + c_n(x) g(n, j) + \Delta_j\left(h(n, j)\right),
\end{equation}
where
\begin{align*}
& h(n, j) = -\frac{3 (j - 1) j (2 j - 2 n - 1)}{2 (j - n -3) (j - n - 2) (j - n - 1)} g(n, j), \\
& \Delta_j\left(h(n, j)\right)  = h(n, j + 1) - h(n, j),
\end{align*}
for $j \geq n + 4$. Taking sum in~\eqref{EQ:hypergeo} with respect to $j$ from $n + 4$ to $\infty$, we get 
\begin{multline}
4 (n + 2) (n + 3) \sum_{j = n + 4}^\infty g(n + 3, j) = a_n(x)  \sum_{j = n + 4}^\infty  g(n + 2, j) + b_n(x) \sum_{j = n + 4}^\infty  g(n + 1, j) \\ 
+ c_n(x) \sum_{j = n + 4}^\infty  g(n, j) + \lim_{j \rightarrow \infty} h(n, j) - h(n, n + 4). \label{EQ:sum}
\end{multline}
Note that 
\begin{align*}
& \sum_{j = n + 4}^\infty g(n + 3, j) = s_{n + 3}(x), \\
& \sum_{j = n + 4}^\infty  g(n + 2, j) = s_{n + 2}(x) - g(n + 2, n + 3), \\
& \sum_{j = n + 4}^\infty  g(n + 1, j) = s_{n + 1}(x) - [g(n+1, n+2) + g(n + 1, n+3)], \\
& \sum_{j = n + 4}^\infty  g(n, j) = s_{n}(x) - [g(n, n+1) + g(n,n +2) + g(n, n+3)], \\
& \lim_{j \rightarrow \infty} h(n, j) = 0. 
\end{align*}
Thus, we see that~\eqref{EQ:sum} leads to~\eqref{EQ:rec}.

By the {\tt Mathematica} command ``{\tt Sum}", it is straightforward to find the closed formulae~\eqref{EQ:s0} and~\eqref{EQ:s1} for $s_0(x)$ and $s_1(x)$, respectively. 

To derive a formula for $s_2(x)$ consider
\begin{align*}
s_2(x) & = \sum_{j = 3}^\infty \frac{(-1/2)_j (1/2 - j) (3/2 - j)}{j! (1 - j) (2 - j)} (-x)^j \\
           & = \sum_{j = 3}^\infty \frac{(-1/2)_j (1/2 - j)}{j! (1 - j)} (-x)^j - \frac{1}{2} \sum_{j = 3}^\infty \frac{(-1/2)_j (1/2 - j)}{j! (1 - j) (2 - j)} (-x)^j \\
           & = s_1(x) + \frac{3}{16} x^2 - \frac{1}{2} \sum_{j = 3}^\infty \frac{(-1/2)_j (1/2 - j)}{j! (1 - j) (2 - j)} (-x)^j \\
           & = s_1(x) + \frac{3}{16} x^2 - \frac{x^2}{2} \sum_{j = 3}^\infty \frac{(-1/2)_j (1/2 - j)}{j! (1 - j) (2 - j)} (-x)^{j - 2}. 
\end{align*}
Set 
$$
f(x) = \sum_{j = 3}^\infty \frac{(-1/2)_j (1/2 - j)}{j! (1 - j) (2 - j)} (-x)^{j - 2}. 
$$
Then $s_2(x) = s_1(x) + \frac{3}{16}x^2 - \frac{1}{2}x^2f(x)$. Thus, in order to derive~\eqref{EQ:s2}, we just need  a closed formula for $f(x)$. 
Note that 
\begin{align*}
(-x)^3 f'(x) & = \sum_{j = 3}^\infty \frac{(-1/2)_j (1/2 - j)}{j! (1 - j)} (-x)^j \\
                   & = s_1(x) + \frac{3}{16} x^2.
\end{align*}
Thus, we have 
\[
f(x) = \int_{0}^x \left(-\frac{1}{t^3} s_1(t) - \frac{3}{16 t}\right) dt.
\]
Using  the {\tt Maple} command ``{\tt int}", we obtain a closed formula for $f(x)$, which leads to~\eqref{EQ:s2}.
\end{proof}

Next, we present the main result of this section.

\begin{theorem} \label{THM:firstexpansion}
For each $(\lambda, k) \in (0, 1) \times (0, 1)$ and an integer $N \geq 1$, the second incomplete EI admits the following representation
\begin{align}
E(\lambda, k) & = \lambda \sqrt{1 + \frac{\lambda^2}{\beta}} + \ln\left(\frac{1 - \lambda}{1 + \lambda}\right) \sum_{j = 1}^N \frac{(-1/2)_j (1/2)_j}{j! (j - 1)!} (1 - k^2)^j \nonumber \\
 & \quad -\frac{1}{\lambda} \sum_{n = 0}^{N - 1} \left(\frac{1 - \lambda^2}{-\lambda^2} \right)^n s_n\left( \frac{\lambda^2}{\beta}\right) + R_N(\lambda, k), \label{EQ:EIfirstexpansion}
\end{align}
where $\beta$ is defined in \eqref{eq:beta-defined} and the function $s_n(x)$ is given in Lemma~\ref{LEM:ct}. Moreover, the remainder $R_N(\lambda, k)$ is negative and satisfies
\begin{equation} \label{EQ:remainderbound1}
\frac{(1/2)_N (1/2)_{N+1} (1 - k^2)^N}{2 N! (N+ 1)!} f_{N+1}(\lambda, k) < - R_{N}(\lambda, k)  < \frac{(1/2)_N (1/2)_{N+1} (1 - k^2)^N}{2 N! (N + 1)!} f_{N}(\lambda, k),
\end{equation}
where the positive function
\begin{equation}\label{EQ:positivefunction}
f_N(\lambda, k) = \frac{1}{1-(1-k^2)/\theta_{N}} 
\Bigg[\frac{\theta_{N}}{\sqrt{\lambda^2 + \beta\theta_{N}}} \ln\frac{\sqrt{\lambda^2+\beta\theta_{N}}+\lambda}{\sqrt{\lambda^2 +\beta\theta_{N}}-\lambda}
+(1-k^2)\ln\frac{1-\lambda}{1+\lambda}\Bigg]
\end{equation} 
with
\begin{equation} \label{EQ:thetaN}
1<\theta_N=\frac{N(N+1)}{(N-1/2)(N+1/2)}\le \frac{8}{3}
\end{equation}
is strictly decreasing in $N$ and bounded on each subset $F$ of the unit $(\lambda,k)$ square such that
\begin{equation}\label{EQ:boundedcond}
\sup_{\lambda, k\in F} \frac{1 - k}{1 - \lambda} < \infty.
\end{equation}
\end{theorem}

\begin{remark}
The error bound~\eqref{EQ:remainderbound1} implies that  expansion~\eqref{EQ:EIfirstexpansion} is convergent at any point in the open unit square and convergence is uniform on compact subsets. Furthermore, it is asymptotic as $k\to1$ along any curve $\gamma$ lying inside the unit square with~\eqref{EQ:boundedcond} satisfied, including those with the endpoint $(1, 1)$.
\end{remark}

\begin{remark}
Condition~\eqref{EQ:boundedcond} is true for any trapezoid with vertices $(0,0)$, $(\alpha,0)$, $(1,1)$, $(0,1)$ for all $0<\alpha<1$.
If condition~\eqref{EQ:boundedcond} is violated but $(1 - k)^n/(1 - \lambda)$ remains bounded, then the $n$-th and 
higher order approximations are still asymptotic. In this case, however, it is much more efficient to employ  expansion~\eqref{EQ:EIsecondexpansion} from Theorem~\ref{THM:secondexpansion}.
\end{remark}
\begin{remark}
Decomposition \eqref{EQ:EIfirstexpansion} together with the inequalities \eqref{EQ:remainderbound1} clearly imply a sequence of asymptotically precise two-sided bounds for the second incomplete EI in the form 
\begin{multline*}
E_N(\lambda, k)-\frac{(1/2)_N(1/2)_{N+1}(1-k^2)^N}{2N!(N+1)!}f_{N}(\lambda, k) \le E(\lambda, k)
\\
\le E_N(\lambda,k)-\frac{(1/2)_N (1/2)_{N+1}(1-k^2)^N}{2N!(N+1)!}f_{N+1}(\lambda, k),
\end{multline*}
where $E_N(\lambda, k)=E(\lambda, k)-R_N(\lambda, k)$ in  \eqref{EQ:EIfirstexpansion} and $f_{N}(\lambda, k)$ is defined in \eqref{EQ:positivefunction}.  Furthermore, we can get a substantially more precise approximation than  $E_N(\lambda,k)$ by substituting $R_N(\lambda, k)$ in \eqref{EQ:EIfirstexpansion}  with its approximate value read off  \eqref{EQ:remainderbound1}. For instance, we can take 
\begin{equation}\label{eq:Eapprox1}
\hat{E}_N(\lambda, k)=E_N(\lambda, k)-\frac{(1/2)_N (1/2)_{N+1}(1-k^2)^N}{2N!(N+1)!}f_{N+\varepsilon}    
\end{equation}
with some $\varepsilon\in(0,1)$. Then it follows from~\eqref{EQ:remainderbound1} that the corresponding remainder $\hat{R}_N(\lambda, k) := E(\lambda, k) - \hat{E}_N(\lambda, k)$ 
satisfies 
\begin{multline} \label{EQ:refinedremainderbound1}
\frac{(1/2)_N (1/2)_{N+1} (1 - k^2)^N}{2 N! (N+ 1)!} \left(f_{N+\varepsilon}(\lambda, k) - f_{N}(\lambda, k)\right) < \hat{R}_{N}(\lambda, k)  \\
< \frac{(1/2)_N (1/2)_{N+1} (1 - k^2)^N}{2 N! (N + 1)!} \left(f_{N+\varepsilon}(\lambda, k)  - f_{N+ 1}(\lambda, k)\right).
\end{multline}
We use the value $\varepsilon=1/2$ in Section~5 for numerical experiments.  It seems to be an interesting open problem to find the value of $\varepsilon=\varepsilon(a)$ giving the best approximant \eqref{eq:Eapprox1} in the uniform norm on the subset of the unit square of the form $(a,1)\times(a,1)$.     
\end{remark}

\begin{proof}
By Theorem~\ref{THM:first}, for $\lambda$ and $k$ satisfying~\eqref{EQ:cond1}, the incomplete EI of the 
second kind has the following expansion
\begin{align*}
E(\lambda, k) & = \lambda  \sqrt{1 + \frac{(1 - k^2) \lambda^2}{1 - \lambda^2}} 
 + \ln\left(\frac{1 - \lambda}{1 + \lambda}\right) \sum_{j = 1}^\infty \frac{(-1/2)_j (1/2)_j}{j! (j - 1)!} (1 - k^2)^j  \\
& \quad + \frac{1}{\lambda} \sum_{j = 1}^\infty \sum_{n = 0}^{j -1} (-1)^{j + n - 1} \frac{(-1/2)_j (1/2 - j)_n}{j! (1 - j)_n} (1-k^2)^j \left( \frac{\lambda^2}{1 - \lambda^2}\right)^{j - n}. 
\end{align*}
Using the rearrangement rule
\begin{equation} \label{EQ:exchangerule}
\sum_{j = N + 1}^\infty \sum_{n = N}^{j - 1} b_{n, j} = \sum_{n = N}^\infty \sum_{j = n + 1}^\infty b_{n, j}
\end{equation}
for $N = 0$, we get
\begin{align}
E(\lambda, k) & = \lambda  \sqrt{1 + \frac{(1 - k^2) \lambda^2}{1 - \lambda^2}} 
 + \ln\left(\frac{1 - \lambda}{1 + \lambda}\right) \sum_{j = 1}^\infty \frac{(-1/2)_j (1/2)_j}{j! (j - 1)!} (1 - k^2)^j  \nonumber \\
& \quad + \frac{1}{\lambda} \sum_{n = 0}^\infty \sum_{j = n + 1}^\infty (-1)^{j + n - 1} \frac{(-1/2)_j (1/2 - j)_n}{j! (1 - j)_n} (1-k^2)^j \left( \frac{\lambda^2}{1 - \lambda^2}\right)^{j - n}.  \label{EQ:exchangesum}
\end{align}
Taking the $N$-th partial sum of the first series, the $(N-1)$-th partial sum of the second series in~\eqref{EQ:exchangesum}, and 
writing the rest as a remainder, we get~\eqref{EQ:EIfirstexpansion} by Lemma~\ref{LEM:ct}. Thereby, the reminder is given by
\begin{align}
R_N(\lambda, k) & = \ln\left(\frac{1 - \lambda}{1 + \lambda}\right) \sum_{j = N + 1}^\infty \frac{(-1/2)_j (1/2)_j}{j! (j - 1)!} (1 - k^2)^j  \nonumber \\
& \quad + \frac{1}{\lambda} \sum_{n = N}^\infty \sum_{j = n + 1}^\infty (-1)^{j + n - 1} \frac{(-1/2)_j (1/2 - j)_n}{j! (1 - j)_n} (1-k^2)^j \left( \frac{\lambda^2}{1 - \lambda^2}\right)^{j - n}. \label{EQ:firstremainder1}
\end{align}
Applying the rule~\eqref{EQ:exchangerule} to the second term in~\eqref{EQ:firstremainder1}, we have
\begin{align}
R_N(\lambda, k) & = \ln\left(\frac{1-\lambda}{1+\lambda}\right)\sum_{j=N+1}^\infty \frac{(-1/2)_j (1/2)_j}{j! (j - 1)!} (1 - k^2)^j  \nonumber \\
& \quad - \frac{1}{\lambda} \sum_{j=N+1}^\infty (-1)^j\frac{(-1/2)_j}{j!} \left[\frac{(1- k^2) \lambda^2}{1 - \lambda^2} \right]^j \sum_{n = N}^{j - 1}  \frac{(1/2 - j)_n}{(1 - j)_n} \left( \frac{1 - \lambda^2}{- \lambda^2}\right)^n. \label{EQ:firstremainder2}
\end{align}
By the argument from~\cite[page 197]{Dima2007}, we have that
\begin{multline}
\sum_{n = N}^{j - 1}  \frac{(1/2 - j)_n}{(1 - j)_n} \left( \frac{1 - \lambda^2}{- \lambda^2}\right)^n  = \frac{(1/2)_j}{(j - 1)!} \left(\frac{1-\lambda^2}{-\lambda^2}\right)^j \Bigg[\lambda\ln\left(\frac{1-\lambda}{1+\lambda}\right)\\
+ \int_{0}^{\frac{\lambda^2}{1 - \lambda^2}} \frac{(-u)^{j - N} du}{(1 + u) \sqrt{1 - u (1 - \lambda^2)/\lambda^2}} \Bigg]. \label{EQ:DimaSum}
\end{multline}
Substituting~\eqref{EQ:DimaSum} into~\eqref{EQ:firstremainder2} and interchanging the summation and integration, we get
\begin{equation} 
R_N(\lambda, k) = \frac{(-1)^{N+1}}{\lambda} \int_{0}^{\frac{\lambda^2}{1 - \lambda^2}} \frac{\left(1 - u(1 - \lambda^2)/\lambda^2\right)^{-1/2}}{u^N (1 + u)} du \sum_{j = N + 1}^\infty \frac{(-1/2)_j (1/2)_j}{j! (j - 1)!} [-(1 - k^2) u]^j. \label{EQ:firstremainder3}
\end{equation}
Using the {\tt Mathematica} command ``{\tt Sum}", we find that 
\begin{multline} 
 \sum_{j = N + 1}^\infty \frac{(-1/2)_j (1/2)_j}{j!(j-1)!} (-x)^j = \frac{(-1/2)_{N + 1} (1/2)_{N + 1}}{N! (N + 1)!} (-x)^{N + 1} \\ 
\cdot {}_3F_2(1, N+1/2, N+3/2; N+1, N+2; -x). \label{EQ:firstremainderSum}
\end{multline}
Substituting~\eqref{EQ:firstremainderSum} into~\eqref{EQ:firstremainder3}, we have
\begin{multline}
R_N(\lambda, k) = \frac{(1 - k^2)^{N + 1} (-1/2)_{N+1} (1/2)_{N+1}}{\lambda N! (N+1)!} \\
\cdot  \int_{0}^{\frac{\lambda^2}{1 - \lambda^2}} \frac{{}_3F_2\left(1, N + 1/2, N + 3/2; N + 1, N + 2; - (1 - k^2)u \right) u du}{(1 + u) \left(1 - u (1 - \lambda^2)/\lambda^2 \right)^{1/2}}. \label{EQ:firstremainder4}
\end{multline}
Applying the following inequality~\cite[Theorem 2]{Dima2009}
\begin{multline*}
\frac{1}{1 + \frac{(N + 1/2) (N + 3/2)}{(N + 1) (N + 2)} x} < {}_3F_2(1, N + 1/2, N + 3/2; N + 1, N + 2; - x)  \\
< \frac{1}{1 + \frac{(N - 1/2) (N + 1/2)}{N (N + 1)} x} \quad \text{ for each} \quad x > 0
\end{multline*}
to~\eqref{EQ:firstremainder4}, we see that $R_N(\lambda, k) < 0$ and 
\begin{equation*}
\frac{(1/2)_N (1/2)_{N+1} (1 - k^2)^N}{2 N! (N+ 1)!} g(\theta_{N+1}, \lambda, k) < - R_{N}(\lambda, k)  < \frac{(1/2)_N (1/2)_{N+1} (1 - k^2)^N}{2 N! (N + 1)!} g(\theta_{N}, \lambda, k),
\end{equation*}
where $\theta_N$ is defined in \eqref{EQ:thetaN} and 
\begin{multline}\label{EQ:positivefunction2}
g(\theta, \lambda, k) = \frac{1 - k^2}{\lambda} \int_{0}^{\frac{\lambda^2}{1-\lambda^2}}\frac{udu}{[1+(1-k^2)u/\theta](1+u) (1-u(1-\lambda^2)/\lambda^2)^{1/2}} \\
= \frac{\theta}{\theta-(1- k^2)}\Bigg[\frac{\theta}{\sqrt{\lambda^2 + \beta\theta}}\cdot
\ln\frac{\sqrt{\lambda^2 + \beta\theta} + \lambda}{\sqrt{\lambda^2+\beta\theta}-\lambda}-(1-k^2)\ln\left(\frac{1+\lambda}{1-\lambda}\right) \Bigg]
\end{multline} 
with $\beta$ from \eqref{eq:beta-defined}. Set 
$$
f_N(\lambda, k) = g(\theta_N, \lambda, k). 
$$
Then we get the error bound~\eqref{EQ:remainderbound1}. 
The boundedness of $f_N(\lambda, k)$ under condition~\eqref{EQ:boundedcond} can be deduced from the second representation 
of $g(\theta, \lambda, k)$ in~\eqref{EQ:positivefunction2} while the monotonicity of $f_N(\lambda, k)$ in $N$ follows from the first  representation of $g(\theta, \lambda, k)$ in~\eqref{EQ:positivefunction2}. 

We will now remove condition \eqref{EQ:cond1} and show that  the expansion~\eqref{EQ:EIfirstexpansion} is true  in the entire unit square in the $(\lambda, k)$ plane.  Indeed, as we remarked in the introduction the function $E(\lambda,k)$ is holomorphic in the bi-disk $|\lambda|<1$, $|k|<1$ of $\mathbb{C}^2$.  The same is true for the terms on the right hand side preceding $R_N(\lambda,k)$. Indeed, the hypergeometric representation \eqref{eq:sn4F3} implies that $s_n\left(\frac{(1-k^2)\lambda^2}{1-\lambda^2}\right)$ has singularity at
$$
\frac{(1-k^2)\lambda^2}{1-\lambda^2}=-1~\Leftrightarrow~ k^2\lambda^2=1,
$$
so that $s_n$ in \eqref{EQ:EIfirstexpansion} is also holomorphic in the bi-disk.  Finally, the apparent singularity at $\lambda=0$ is removable because of $(-x)^{n+1}$ in front of ${}_4F_{3}$ in~\eqref{eq:sn4F3}.  On the other hand, 
the remainder $R_N(\lambda,k)$ is holomorphic in the same bi-disk due to representation~\eqref{EQ:firstremainder4}.  Hence, the difference of $E(\lambda,k)$ and the terms on the right hand side of \eqref{EQ:EIfirstexpansion} preceding $R_N(\lambda,k)$  coincide with  $R_N(\lambda,k)$ under condition \eqref{EQ:cond1}.  The principle of analytic continuation implies that they coincide in the entire bi-disk. 
\end{proof}

\begin{remark}
By the above proof, we see that the expansion~\eqref{EQ:EIfirstexpansion} also holds for complex $\lambda$ and $k$ in the bi-disk. 
\end{remark}

By~\eqref{EQ:EIfirstexpansion}, we obtain the following first order approximation for the incomplete EI of the second kind (see details in Section~5)
$$
E_1(\lambda, k)= (\lambda-1/\lambda)\sqrt{1 + \lambda^2/\beta} - \frac{1-k^2}{4}\ln\left(\frac{1-\lambda}{1+\lambda}\right) +1/\lambda,
$$
which is not only the correct asymptotic approximation for $E(\lambda, k)$ as $k \rightarrow 1$ but also as $\lambda \to0$ 
including the case when both $\lambda, k \to 0$ along any curve inside the unit square. 
In fact, it is straightforward to see that
\[
E_1(\lambda, k) = \lambda + \frac{1}{24}(7 - 10 k^2 + 3 k^4) \lambda^3 + \mathcal{O}(\lambda^5), \quad \lambda \to 0.
\]
On the other hand, we have 
\[
E(\lambda, k)=\lambda+\frac{1}{6}(1 - k^2) \lambda^3 + \mathcal{O}(\lambda^5), \quad \lambda \to 0.
\]
Thus, we get
\[
E(\lambda, k) - E_1(\lambda, k) = \mathcal{O}(\lambda^3), \quad \lambda \to 0.
\]
In other words, $E_1(\lambda, k)$ is indeed an approximation for two sides of the unit square (including endpoints), \ie, 
the side $\lambda = 0, k \in [0, 1]$ and the other side $k = 1, \lambda \in [0, 1]$. The same phenomenon happens for 
higher order approximations but the asymptotic order for $\lambda \to 0$ does not increase with $N$.

Approximation \eqref{eq:Eapprox1}  with $\varepsilon=1/2$ takes the form:
$$
\hat{E}_1(\lambda,k)=E_1(\lambda, k)-\frac{3(1 - k^2)}{32} f_{3/2}(\lambda, k),
$$
$$
f_{3/2}(\lambda, k)=\frac{15}{15-8(1-k^2)} 
\Bigg[\frac{15}{8\sqrt{\lambda^2 + 15\beta/8}} \ln\frac{\sqrt{\lambda^2+15\beta/8}+\lambda}{\sqrt{\lambda^2 +15\beta/8}-\lambda}
-(1-k^2)\ln\frac{1+\lambda}{1-\lambda}\Bigg].
$$
\section{The second asymptotic expansion} \label{SEC:sec}

For $n \in \bN$ and $(\lambda, k) \in [0, 1] \times [0, 1)$, set 
\begin{align}
A_n(x) & = \sum_{j = 0}^\infty {n + j \choose j} \frac{(-1)^j (1/2)_j}{(2 (n + j) + 1) j!} x^j, \nonumber \\
B_n(x) & =  \sum_{j = 0}^\infty {n + j \choose j} \frac{(-1)^j (1/2)_j}{(2 (n + j) + 3) j!} x^j, \nonumber \\
C_n(x) & = A_n(x) + \frac{(1 - \lambda^2) k^2}{1 - k^2} B_n(x). \label{EQ:sumab}
\end{align}
We give two representations for $A_n(x)$, $B_n(x)$ in the following lemma. The first one in terms of elementary functions serves as an ingredient of our second expansion.  The second one in terms of the Gauss hypergeometric function is needed for the error estimation in Theorem~\ref{THM:secondexpansion} below.

\begin{lemma} \label{LEM:secondexpansion}
For $n \in \bN$ and $(\lambda, k) \in [0, 1] \times [0, 1)$, we have the following identities for $A_n(x)$ and $B_n(x)$:
\begin{subequations}\label{EQ:ABformulas} 
\begin{align}
A_n(x)  & {=} \frac{1}{n!} D_x^n \left[ (-1)^n \frac{(1/2)_n}{n! \sqrt{x}} \ln(\sqrt{1 + x} {+} \sqrt{x}) {+} \frac{\sqrt{1 + x}}{2 n x} 
\sum_{j = 0}^{n - 1} (-1)^j \frac{(1/2 - n)_j}{(1 - n)_j} x^{n - j} \right] \label{EQ:explicitformulaA} \\
 & = \frac{1}{2 x^{n + 1/2}} \int_0^x \frac{t^{n - 1/2}}{\sqrt{1 + t}} {}_2F_1\left(-n, 1/2; 1; \frac{t}{1 + t} \right)dt,\label{EQ:errorestimationA}
\end{align}
where $D_x$ is the usual differentiation with respect to $x$ and the second term in the first bracket equals zero when $n = 0$; 
and
\begin{align}
B_n(x)  & {=} \frac{1}{n!} D_x^n \Bigg[\frac{(-1)^{n+1}(1/2)_{n + 1}}{(n + 1)! x^{3/2}} \ln(\sqrt{1 + x} + \sqrt{x})  + \frac{\sqrt{1 + x}}{2 (n + 1) x^2} 
\sum_{j = 0}^{n}  \frac{(-1/2 - n)_j}{(-1)^j(- n)_j} x^{n + 1 - j} \Bigg] \label{EQ:explicitformulaB} \\ 
& =\frac{1}{2 x^{n + 3/2}}  \int_0^x \frac{t^{n + 1/2}}{\sqrt{1 + t}} {}_2F_1\left(-n, 1/2; 1; \frac{t}{1 + t} \right) dt.\label{EQ:errorestimationB}
\end{align} 
\end{subequations}
\end{lemma}

\begin{proof}
The identities \eqref{EQ:explicitformulaA} and \eqref{EQ:errorestimationA} for $A_n(x)$ are given in~\cite[Lemma 4]{Dima2007}. Hence, we only need to 
derive the expressions for $B_n(x)$. 

On account of $1/\sqrt{1 + x} = \sum_{j = 0}^\infty (-1)^j (1/2)_j/j! x^j$, we see that 
\begin{align}
B_n(x) & =  \sum_{j = 0}^\infty {n + j \choose j} \frac{(-1)^j (1/2)_j}{(2 (n + j) + 3) j!} x^j 
 = \frac{1}{2n!} D_x^n x^{-3/2} \int_0^x \frac{t^{n + 1/2}}{\sqrt{1 + t}} dt. \label{EQ:bint}
\end{align}
For the integral on the right side of~\eqref{EQ:bint}, we make a change of variables by $y^2 = t/(1 + t)$ and then get
\begin{equation} \label{EQ:bch1}
\int_0^x \frac{t^{n + 1/2}}{\sqrt{1 + t}} dt = 2 \int_{0}^{\sqrt{x/(1 + x)}} \frac{y^{2 (n + 1)} dy}{(1 - y^2)^{n + 2}}. 
\end{equation}
By~\eqref{EQ:bint},~\eqref{EQ:bch1}, and~\cite[Lemma 1 and 4]{Dima2007}, we derive~\eqref{EQ:explicitformulaB}. 

Alternatively, we set $t=ux$ on the right hand side of~\eqref{EQ:bint} and get
\[
\int_0^x \frac{t^{n + 1/2}}{\sqrt{1 + t}} dt = x^{n + 3/2} \int_{0}^1 \frac{u^{n + 1/2} du}{\sqrt{1 + u x}}. 
\] 
Thus, we have
\begin{align*}
B_n(x) & = \frac{1}{2n!} D_x^n x^{-3/2} \int_0^x \frac{t^{n + 1/2}}{\sqrt{1 + t}} dt  = \frac{1}{2n!} D_x^n x^n \int_{0}^1 \frac{u^{n + 1/2} du}{\sqrt{1 + u x}} \\
& =  \frac{1}{2n!} \int_{0}^1 u^{n + 1/2} \left[ D_x^n \frac{x^n}{\sqrt{1 + u x}} \right] du = \frac{1}{2} \int_{0}^1 \frac{u^{n + 1/2}}{\sqrt{1 + u x}} {}_2F_1\left(-n, 1/2; 1; \frac{u x}{1 + u x}\right) du.
\end{align*}
Finally, substituting back $t=ux$, we obtain
\begin{equation} \label{EQ:bint2}
B_n(x) = \frac{1}{2 x^{n + 3/2}} \int_{0}^x \frac{t^{n + 1/2}}{\sqrt{1 + t}} {}_2F_1\left(-n, 1/2; 1; \frac{t}{1 + t}\right) dt.
\end{equation}
Therefore, it follows from~\eqref{EQ:bint2} and~\cite[Lemma 1 and 4]{Dima2007} that~\eqref{EQ:errorestimationB} holds.
\end{proof}

\begin{theorem} \label{THM:secondexpansion}
For each $(\lambda, k) \in (0, 1) \times (0, 1)$ and an integer $N \geq 1$, the second incomplete EI can be decomposed as follows:
\begin{equation} \label{EQ:EIsecondexpansion}
E(\lambda,k)=E(k)-\sqrt{(1-\lambda^2)(1-k^2)}\sum_{n=0}^{N-1} (1-\lambda^2)^n C_n\left(\beta\right) + \tilde{R}_N(\lambda, k), 
\end{equation}
where $\beta=(1-\lambda^2)/(1-k^2)$ as before and the function $C_n(x)$ is defined in \eqref{EQ:sumab} and computed in Lemma~\ref{LEM:secondexpansion}. The remainder is negative and satisfies the following 
inequalities:
\begin{multline} 
\frac{(1-\lambda^2)^{N+1}(\lambda^2+\beta+N^{-1})(1/2)_N}{2\beta^2 (N+1)!}\
\Big[\sqrt{\beta\left(1+\beta\right)}-\mathrm{arcsinh}\!\left(\sqrt{\beta}\right) \Big]
\\
\le-\tilde{R}_N(\lambda, k)\le\frac{(1-\lambda^2)^{N+1}(\lambda^2+\beta+N^{-1})}{2(N+1)\lambda^{2}\sqrt{\beta(1+\beta)}}.
\label{EQ:secondEIremainderbound}
\end{multline}
\end{theorem}

\begin{remark}
The error bound~\eqref{EQ:secondEIremainderbound} implies that expansion ~\eqref{EQ:EIsecondexpansion} is convergent for any fixed $(\lambda, k)$ in the open unit square and convergence is uniform on compact subsets. Furthermore, it is asymptotic as $\lambda\to1$ along any curve lying entirely inside the unit square, including those with the endpoint $(1,1)$. 
\end{remark}
\begin{remark}
Decomposition \eqref{EQ:EIsecondexpansion} together with the inequalities \eqref{EQ:secondEIremainderbound} clearly implies a sequence of asymptotically precise two-sided bounds for the second incomplete EI in the form 
\begin{multline*}
\tilde{E}_N(\lambda, k)-\frac{(1-\lambda^2)^{N+1}(\lambda^2+\beta+N^{-1})}{2(N+1)\lambda^{2}\sqrt{\beta(1+\beta)}}
\le E(\lambda, k)
\\
\le \tilde{E}_N(\lambda, k)-\frac{(1-\lambda^2)^{N+1}(\lambda^2+\beta+N^{-1})(1/2)_N}{2\beta^2 (N+1)!}\
\Big[\sqrt{\beta\left(1+\beta\right)}-\mathrm{arcsinh}\!\left(\sqrt{\beta}\right) \Big],
\end{multline*}
where $\tilde{E}_N(\lambda, k)=E(\lambda, k)-\tilde{R}_N(\lambda, k)$ in  \eqref{EQ:EIsecondexpansion}. Furthermore, we can get a substantially more precise approximation than  $\tilde{E}_N(\lambda,k)$ by substituting $\tilde{R}_N(\lambda, k)$ in \eqref{EQ:EIsecondexpansion}  with its approximate value. For instance,  we can take the weighted average 
\begin{multline} \label{eq:Eapprox2}
\bar{E}_N(\lambda, k) = \tilde{E}_N(\lambda, k)-\frac{(1 - \lambda^2)^{N + 1} (\lambda^2 + \beta + N^{-1})}{2 (N + 1)} \left(\frac{\delta}{\lambda^2 \sqrt{\beta (1 + \beta)}} \right. \\
\left. + (1-\delta)\frac{(1/2)_N \left[\sqrt{\beta (1 + \beta)}  -\mathrm{arcsinh}\!\left(\sqrt{\beta}\right) \right]}{\beta^2 N!} \right),
\end{multline}
with $0<\delta<1$.  Numerical experiments in section~5 give the optimal value of $\delta=\frac{67}{187}$. Set
\begin{equation}   \label{EQ:errordiff2}
\bar{\Delta}_N  = \frac{(1 - \lambda^2)^{N + 1} (\lambda^2 + \beta + N^{-1})}{2 (N + 1) E(\lambda, k)} \left(\frac{1}{\lambda^2 \sqrt{\beta (1 + \beta)}} -  \frac{(1/2)_N \left[\sqrt{\beta (1 + \beta)}  -\mathrm{arcsinh}\!\left(\sqrt{\beta}\right) \right]}{\beta^2 N!} \right).
\end{equation}
Then it follows from~\eqref{EQ:secondEIremainderbound} that the corresponding remainder $\bar{R}_N(\lambda, k) := E(\lambda, k) - \bar{E}_N(\lambda, k)$ satisfies 
\begin{equation} \label{EQ:refinedsecondEIremainderbound}
-(1-\delta)\bar{\Delta}_N \cdot E(\lambda, k)  \leq \bar{R}_N(\lambda, k) \leq \delta\bar{\Delta}_N\cdot E(\lambda, k).
\end{equation}
\end{remark}

\begin{proof}
By Theorem~\ref{THM:sec} for $(\lambda, k)$ satisfying~\eqref{EQ:cond2}, we have
\begin{align}
E(\lambda, k) & = \!E(k) 
-\sqrt{(1 - \lambda^2) (1 - k^2)} 
\sum_{m = 0}^{\infty} (1 - \lambda^2)^m \!\left(\frac{1}{2m+1}+\frac{\beta k^2}{2m+3}\right)  {_2}F_1\!\left(\!\!\!\left.\begin{array}{c}-m, 1/2\\ 1\end{array}\right| \frac{1}{1 - k^2}\right) \nonumber \\
& \quad 
=E(k)-\sum_{m = 0}^\infty(1-\lambda^2)^{m+1/2}\left(\frac{1}{2m+1} +\frac{\beta k^2}{2m+3}\right)\sum_{j = 0}^m {m \choose j} \frac{(-1)^j (1/2)_j}{j! (1 - k^2)^{j - 1/2}}. \label{EQ:secondEIex1}
\end{align}
Applying the rule
\[
\sum_{m = 0}^\infty \sum_{j = 0}^m b_{m, j} = \sum_{n = 0}^\infty \sum_{j = 0}^\infty b_{n + j, j}
\]
to~\eqref{EQ:secondEIex1}, we get 
\begin{align}
E(\lambda, k) & {=} E(k) {-} \sqrt{(1{-}\lambda^2)(1 {-} k^2)}  \sum_{n = 0}^\infty (1 {-} \lambda^2)^n 
\nonumber\\ 
& \quad \ \times\sum_{j = 0}^\infty 
\beta^j {n + j \choose j} \frac{(-1)^j (1/2)_j}{j!} 
 \left(\frac{1}{2(n+j)+1}+\frac{\beta{k^2}}{2 (n + j) + 3}\right), \label{EQ:secondEIex2}
\end{align}
which, in view of~\eqref{EQ:sumab}, can be split as follows:
\[
E(\lambda, k) = E(k) - \sqrt{(1 {-} \lambda^2) (1 {-} k^2)} \sum_{n = 0}^{N - 1} (1 - \lambda^2)^{n} C_n\left(\beta\right) + \tilde{R}_N(\lambda, k),
\]
where the remainder is 
\begin{equation} \label{EQ:secondEIremainder}
\tilde{R}_N(\lambda, k) =  - \sqrt{(1 {-} \lambda^2) (1 {-} k^2)} \sum_{n = N}^{\infty} (1 - \lambda^2)^{n} C_n\left(\beta\right).
\end{equation}
In~\eqref{EQ:secondEIex2}, the inner sum does not converge unless $k<\lambda$. Nevertheless, it follows from~\eqref{EQ:ABformulas} that the function $C_n(x)$ is an elementary function with no singularities in the unit square $(\lambda,k)\in[0,1]\times[0,1)$, which furnishes analytic continuation of the inner sum to this domain. Next, we prove that the outer sum of~\eqref{EQ:secondEIex2} converges for each $(\lambda, k)$ in the unit square. For this purpose,  substituting~\eqref{EQ:errorestimationA} and \eqref{EQ:errorestimationB} into~\eqref{EQ:secondEIremainder} and applying the upper bound from item~(ii) of Lemma~\ref{LEM:sec}, we get
\begin{align}
-\tilde{R}_N(\lambda, k) & = \frac{1}{2} \sum_{n = N}^\infty (1 - k^2)^{n + 1} \int_0^{\frac{1 - \lambda^2}{1 - k^2}}  \frac{t^n}{\sqrt{t (1 + t)}} {}_2F_1\left(-n, 1/2; 1; \frac{t}{1 + t} \right) dt  \nonumber \\
\quad \ +  \frac{k^2}{2} & \sum_{n = N}^\infty (1 - k^2)^{n + 1} \int_0^{\frac{1 - \lambda^2}{1 - k^2}}
 t^n \sqrt{\frac{t}{1 + t}} {}_2F_1\left(-n, 1/2; 1; \frac{t}{1 + t} \right) dt  \nonumber \\
 \leq  \frac{1}{2}  &\sum_{n = N}^\infty(1 - k^2)^{n + 1} \int_0^{\frac{1 - \lambda^2}{1 - k^2}}  \frac{t^n dt}{\sqrt{t (1 + t)}}+ \frac{k^2}{2} \sum_{n = N}^\infty (1 - k^2)^{n + 1} \int_0^{\frac{1 - \lambda^2}{1 - k^2}}
 t^n \sqrt{\frac{t}{1 + t}} dt  \nonumber \\
= \frac{1 - k^2}{2} &\int_0^{\frac{1 - \lambda^2}{1 - k^2}}  \frac{dt}{\sqrt{t (1 + t)}} \sum_{n = N}^\infty [(1 - k^2) t]^n 
 +  \frac{(1 - k^2) k^2}{2}  \int_0^{\frac{1 - \lambda^2}{1 - k^2}}
 \sqrt{\frac{t}{1 + t}} dt \sum_{n = N}^\infty  [(1 - k^2) t]^n  \nonumber \\
= \frac{(1 - k^2)^{N + 1}}{2} &\int_0^{\frac{1 - \lambda^2}{1 - k^2}}  \frac{t^N dt}{\sqrt{t (1 + t)}[1 - (1 - k^2) t]} 
 +  \frac{(1 - k^2)^{N + 1} k^2}{2}  \int_0^{\frac{1 - \lambda^2}{1 - k^2}}
 \frac{t^{N + 1/2} dt}{\sqrt{1 + t}[1 - (1 - k^2) t]}  \nonumber \\
 \leq \quad  & \frac{(1 - k^2)^{N + 1}}{2 \lambda^2} \int_0^{\frac{1 - \lambda^2}{1 - k^2}}  \frac{t^N dt}{\sqrt{t (1 + t)}} 
+  \frac{(1 - k^2)^{N + 1} k^2}{2 \lambda^2}  \int_0^{\frac{1 - \lambda^2}{1 - k^2}}
 \frac{t^{N + 1/2} dt}{\sqrt{1 + t}}  \nonumber \\
= \frac{(1 - k^2)^{N + 1}}{\lambda^2} &\int_0^{\sqrt{\frac{1 - \lambda^2}{2 - k^2 - \lambda^2}}}  \frac{y^{2 N} dy}{(1 - y^2)^{N + 1}}
+ \frac{(1 - k^2)^{N + 1} k^2}{\lambda^2} \int_0^{\sqrt{\frac{1 - \lambda^2}{2 - k^2 - \lambda^2}}}  \frac{y^{2 (N + 1)} dy}{(1 - y^2)^{N + 2}}, \label{EQ:last}
\end{align}
where the last equality is derived from~\eqref{EQ:bch1}. On the other hand, it follows from~\cite[formula 16]{Dima2007} that 
\begin{equation} \label{EQ:dimainequality}
\int_0^x \frac{t^{2 a} dt}{(1 - t^2)^{a + 1}} \leq \frac{x^{2 a + 1}}{2 a (1 - x^2)^a} \quad \text{ for } \quad x \in (0, 1), \quad a > 0.
\end{equation}
By setting $x=(1-\lambda^2)/(2-k^2-\lambda^2)$ and $a=N,N+1$ in~\eqref{EQ:last}, we obtain the upper bound in \eqref{EQ:secondEIremainderbound} by an application of~\eqref{EQ:dimainequality}.

To derive a lower bound, we apply the lower bound from item (ii) of Lemma~\ref{LEM:sec} to representation \eqref{EQ:secondEIremainder} with $C_n$ expressed from Lemma~\ref{LEM:secondexpansion} to get:
\begin{align*}
- \tilde{R}_{N}(\lambda, k) & \geq  \frac{1}{2} \sum_{n = N}^\infty (1 - k^2)^{n + 1} \int_0^{\frac{1 - \lambda^2}{1 - k^2}}  \frac{t^n}{\sqrt{t (1 + t)}}  \frac{(1/2)_n}{n!} dt \\
& \quad \  +  \frac{k^2}{2} \sum_{n = N}^\infty (1 - k^2)^{n + 1} \int_0^{\frac{1 - \lambda^2}{1 - k^2}}
\frac{t^{n + 1/2}}{\sqrt{1 + t}} \frac{(1/2)_n}{n!} dt \\
& = \frac{1 - k^2}{2}  \int_0^{\frac{1 - \lambda^2}{1 - k^2}}   \frac{dt}{\sqrt{t (1 + t)}} \sum_{n = N}^\infty  \frac{(1/2)_n}{n!} 
[(1 - k^2) t]^{n}  \\
& \quad \  +   \frac{(1 - k^2) k^2}{2}  \int_0^{\frac{1 - \lambda^2}{1 - k^2}}  \sqrt{\frac{t}{1 + t}} dt \sum_{n = N}^\infty  \frac{(1/2)_n}{n!} 
[(1 - k^2) t]^{n} \\
& {=} \frac{(1 {-} k^2)^{N + 1} (1/2)_N}{2 N!}  \int_0^{\frac{1 {-} \lambda^2}{1 {-} k^2}}\frac{t^N}{\sqrt{t (1 {+} t)}} {}_2F_1(N {+} 1/2, 1; N {+} 1; (1 {-} k^2) t) dt\\
& {+} \frac{k^2 (1 {-} k^2)^{N + 1}  (1/2)_N}{2 N!}  \int_0^{\frac{1 {-} \lambda^2}{1 {-} k^2}}\frac{t^{N + 1/2}}{\sqrt{1 {+} t}} {}_2F_1(N {+} 1/2, 1; N {+} 1; (1 {-} k^2)t)dt.
\end{align*}
Moreover, it follows from~\cite[Theorem 1.10]{Ponnusamy1997} that 
\[
{}_2F_1(N + 1/2, 1; N + 1; y) \geq \frac{1}{\sqrt{1 - y}} \quad \text{ for } \quad y \in (0, 1).
\]
Therefore, we have
\begin{align*}
- \tilde{R}_{N}(\lambda, k) & \geq  \frac{(1 - k^2)^{N + 1} (1/2)_N}{2 N!}  \int_0^{\frac{1 - \lambda^2}{1 - k^2}}\frac{t^{N - 1/2}}{\sqrt{(1 + t) (1 - (1 - k^2)t)}} dt\\
& \quad \ + \frac{k^2 (1 - k^2)^{N + 1} (1/2)_N}{2 N!}  \int_0^{\frac{1 - \lambda^2}{1 - k^2}}\frac{t^{N + 1/2}}{\sqrt{(1 + t) (1 - (1 - k^2)t)}} dt \\
 & \geq  \frac{(1 - k^2)^{N + 1} (1/2)_N}{2 N!}  \int_0^{\frac{1 - \lambda^2}{1 - k^2}} t^{N - 1} \sqrt{\frac{t}{1 + t}} dt\\
& \quad \ + \frac{k^2 (1 - k^2)^{N + 1} (1/2)_N}{2 N!}  \int_0^{\frac{1 - \lambda^2}{1 - k^2}}t^{N} \sqrt{\frac{t}{1 + t}} dt.
\end{align*}
Using the Chebyshev inequality~\cite[formula IX(1.2)]{Fink1993}, we get
\begin{multline*}
- \tilde{R}_{N}(\lambda, k) \ge \frac{(1 - k^2)^{N + 2} (1/2)_N}{2 (1 - \lambda^2) N!}  \int_0^{\frac{1 - \lambda^2}{1 - k^2}} t^{N - 1} dt   \int_0^{\frac{1 - \lambda^2}{1 - k^2}} \sqrt{\frac{t}{1 + t}} dt\\
+ \frac{k^2 (1 - k^2)^{N + 2} (1/2)_N}{2 (1 - \lambda^2) N!}  \int_0^{\frac{1 - \lambda^2}{1 - k^2}} t^{N} dt \int_0^{\frac{1 - \lambda^2}{1 - k^2}} \sqrt{\frac{t}{1 + t}} dt \\
= \frac{(1 - k^2)^{N+2} (1/2)_N}{2 (1-\lambda^2) N!}\Bigg[\frac{1}{N}\left(\frac{1-\lambda^2}{1-k^2}\right)^{N}+\frac{k^2}{N+1}\left(\frac{1-\lambda^2}{1-k^2}\right)^{N+1}\Bigg]\int_0^{\frac{1-\lambda^2}{1-k^2}} \sqrt{\frac{t}{1 + t}}dt
\\
 = \frac{(1 - k^2)(1-\lambda^2)^{N-1}(1/2)_N}{2N(N+1)!}\big[1-k^2+N(1-k^2\lambda^2)\big]
\Bigg[\sqrt{\frac{1-\lambda^2}{1-k^2}\left(1+\frac{1-\lambda^2}{1-k^2}\right)}-\mathrm{arcsinh}\!\left(\sqrt{\frac{1-\lambda^2}{1-k^2}}\right) \Bigg]
\\
 = \frac{(1-\lambda^2)^{N+1}(1/2)_N}{2\beta^2 (N+1)!}\left(\lambda^2+\beta+N^{-1}\right)\
\Big[\sqrt{\beta\left(1+\beta\right)}-\mathrm{arcsinh}\!\left(\sqrt{\beta}\right) \Big].
\end{multline*}
As the expansion~\eqref{EQ:EIsecondexpansion} holds for $(\lambda, k)$ satisfying~\eqref{EQ:cond2} and both sides are  holomorphic for each $(\lambda, k)$ in the bi-disk $| \lambda | < 1, | k | < 1$ of $\bC^2$. It follows from  the principle of analytic continuation that the expansion~\eqref{EQ:EIsecondexpansion} holds in the entire bi-disk.
\end{proof}

\begin{remark}
By the above proof, we see that expansion~\eqref{EQ:EIsecondexpansion} also holds for complex $\lambda$ and $k$ in the bi-disk. 
\end{remark}

\section{Numerical experiments} \label{SEC:num}

In this section we will give several numerical examples of computations with the asymptotic expansions derived in Section~\ref{SEC:first} and~\ref{SEC:sec}.  First, we consider expansion~\eqref{EQ:EIfirstexpansion}. By~\eqref{EQ:s0} and~\eqref{EQ:s1}, we have ($\beta=(1-\lambda^2)/(1-k^2)$):
\begin{align*}
s_0\left(\frac{\lambda^2}{\beta}\right) & = \sqrt{1 + \lambda^2/\beta} - 1,  \\ 
s_1\left(\frac{\lambda^2}{\beta}\right) & = \frac{1}{2}\sqrt{1+\lambda^2/\beta}-\frac{\lambda^2}{4\beta}\Big[2\ln\frac{1+\sqrt{1+\lambda^2/\beta}}{2}+1\Big]-\frac{1}{2}.
\end{align*}
Therefore, the first and the second order approximations are:
\begin{align}
E_1(\lambda, k) & = (\lambda-1/\lambda)\sqrt{1 + \lambda^2/\beta} - \frac{1-k^2}{4}\ln\!\left(\frac{1-\lambda}{1+\lambda}\right) +1/\lambda,  \label{EQ:firstapprox1} \\ 
E_2(\lambda, k) & = E_1(\lambda, k) - \frac{3}{32} (1 - k^2)^2 \ln\!\left(\frac{1 - \lambda}{1 + \lambda}\right) 
+\frac{1-\lambda^2}{2\lambda^3}\left[\sqrt{1+\lambda^2/\beta}-1\right] \nonumber \\
& \quad \ -\frac{1-k^2}{4\lambda}\Big[2\ln\frac{1+\sqrt{1+\lambda^2/\beta}}{2}+1\Big].\label{EQ:secondapprox1}
\end{align}
The refined approximations \eqref{eq:Eapprox1} with $\varepsilon=1/2$ take the form 
\begin{equation} \label{EQ:refinedfirstapprox}
\hat{E}_1(\lambda, k)=E_1(\lambda, k)-\frac{3(1-k^2)}{32}f_{3/2},
\end{equation}
\begin{equation} \label{EQ:refinedsecondapprox}
\hat{E}_2(\lambda, k)=E_2(\lambda, k)-\frac{15(1-k^2)^2}{256}f_{5/2}, 
\end{equation}
where $f_{N}=f_N(\lambda, k)$ is defined in \eqref{EQ:positivefunction}.

Denote by $\Delta_N$ the range for the relative error defined as the difference between the upper and the lower bounds in~\eqref{EQ:remainderbound1} divided by~$E(\lambda, k)$:
\begin{equation} \label{EQ:errordiff1}
\Delta_N = \frac{(1/2)_N (1/2)_{N+1} (1 - k^2)^N}{2 N! (N+ 1)! E(\lambda, k)} (f_N(\lambda, k) - f_{N+1}(\lambda, k)).
\end{equation}

Approximation~\eqref{EQ:refinedfirstapprox} together with~\eqref{EQ:refinedremainderbound1} places the value of $E(\lambda, k)$ within an interval of 
length $\Delta_1 \cdot E(\lambda, k)$, while~\eqref{EQ:refinedsecondapprox} places $E(\lambda, k)$ within an interval of length $\Delta_2 \cdot E(\lambda, k)$. The results of numerical computations are presented in Table~\ref{TABLE:1}. The exact values of $E(\lambda, k)$ shown in the tables below 
are computed using {\tt Mathematica} with the required number of precise digits guaranteed. 
\begin{table}[h!]
\small
\begin{tabular}{l l l l l l l l}
 \hline
 $\lambda$ & $k$ & $E(\lambda, k)$  & First order & First order & Relative  &  Relative &  Relative error  \\
& &  &  approx.~\eqref{EQ:firstapprox1} &  approx.~\eqref{EQ:refinedfirstapprox} & error $e_1$ &  error $\hat{e}_1$   &  range $\Delta_1$  \\
 \hline
.8 & .8 & .8501 & .8714 & .8496 & $-.02504$ & $.6011 \times 10^{-3}$  & .002446   \\
.9 & .9 & .9504 & .9669 & .9500 &  $-.01734$ & $.4455 \times 10^{-3}$  & .001972  \\
.95 & .95 & .9900 & 1.0003 & .9897 & $-.01044$ & $.2712 \times 10^{-3}$ & .001250   \\
.99 & .99 & 1.0056 & $1.0081$ & 1.0055 & $-.002531$ & $.6475 \times 10^{-4}$  & $.3072 \times 10^{-3}$   \\
.95 & .99 & .9586 & .9591 & .9586 & $-.5674 \times 10^{-3}$ & $.5651\times 10^{-7}$  & $.1743 \times 10^{-4}$ \\ 
.99 & .999 & .9916 & .9916 & .9916 &  $-.3417 \times 10^{-4}$ & $.1902 \times 10^{-8}$   & $.5445 \times 10^{-8}$  \\
\hline
\end{tabular}
\begin{tabular}{l l l l l l l l}
 $\lambda$ & $k$ & $E(\lambda, k)$  & Second order & Second order & Relative & Relative & Relative error  \\
 & & & approx.~\eqref{EQ:secondapprox1} & approx.~\eqref{EQ:refinedsecondapprox} & error $e_2$ & error $\hat{e}_2$  &  range $\Delta_2$   \\
\hline
.8 & .8 & .8501 & .8547 & .8501 & $-.005413$ & $.4975 \times 10^{-4}$  & $.1990 \times 10^{-3}$ \\
.9 & .9 & .9504 & .9523 & .9504 &  $-.001966$ & $.1837 \times 10^{-4}$   & $.8270 \times 10^{-4}$  \\
.95 & .95 & .9900 & .9906 & .9900 & $-.6056 \times 10^{-3}$ & $.5978 \times 10^{-7}$  & $.2661 \times 10^{-4}$  \\
.99 & .99 & 1.0056 & 1.0056 & 1.0056 & $-.2995 \times 10^{-4}$ & $.2667 \times 10^{-8}$  & $.1324 \times 10^{-7}$ \\
.95 & .99 & .9586 & .9586 & .9586 & $-.6968 \times 10^{-7}$ & $.3090 \times 10^{-10}$ &  $.8609 \times 10^{-9}$  \\ 
.99 & .999 & .9916 & .9916 & .9916 & $-.4240 \times 10^{-9}$ & $.1081 \times 10^{-11}$  & $.2810 \times 10^{-11}$  \\
\hline
\end{tabular}
\caption{{\small Numerical examples for approximations~\eqref{EQ:firstapprox1},~\eqref{EQ:refinedfirstapprox}, ~\eqref{EQ:secondapprox1} and~\eqref{EQ:refinedsecondapprox} derived from expansion~\eqref{eq:Eapprox1}. The numbers $e_i$ and $\hat{e}_i$ are the  relative errors $(E(\lambda, k) - {E}_i(\lambda, k))/E(\lambda, k)$ and $( E(\lambda, k) - \hat{E}_i(\lambda, k))/E(\lambda, k)$, respectively, $i = 1, 2$. 
The numbers~$\Delta_1, \Delta_2$ are given in~\eqref{EQ:errordiff1}.}}
\label{TABLE:1}
\end{table}

Next, we consider expansion~\eqref{EQ:EIsecondexpansion}. By~\eqref{EQ:explicitformulaA} and 
\eqref{EQ:explicitformulaB}  we have
\begin{align}
C_0(x) & = \frac{1}{\sqrt{x}} \ln(\sqrt{1 + x} + \sqrt{x}) + \frac{\beta k^2}{2 x^{3/2}}\left(\sqrt{x (1 + x)} - \ln(\sqrt{1 + x} + \sqrt{x})\right), \nonumber \\
C_1(x) & = \frac{1}{4 x}\left(\frac{1}{\sqrt{x}} \ln(\sqrt{1 + x} + \sqrt{x}) - \frac{1 - x}{\sqrt{1 + x}} \right) \nonumber \\ 
& \quad \ + \frac{\beta k^2}{16 x^{2}} \left( \frac{-9}{\sqrt{x}} \ln(\sqrt{1 + x} + \sqrt{x}) + \frac{9 + 3 x + 2 x^2}{\sqrt{1 + x}}\right).  \label{EQ:c1}
\end{align}
Hence, the first and the second order approximation obtained from~\eqref{EQ:EIsecondexpansion} are:
\begin{align}
\tilde{E}_1(\lambda, k) & = E(k) - \frac{2 - 3 k^2 + k^4}{2} \ln\left(\sqrt{1 + \beta} + \sqrt{\beta}\right) 
- \frac{k^2 (1 - k^2)}{2} \sqrt{\beta (1 + \beta)},   \label{EQ:firstapprox2} \\
\tilde{E}_2(\lambda, k) & = \tilde{E}_1(\lambda, k) - (1 - \lambda^2)^{3/2} \sqrt{1 + k} \cdot C_1(\beta), \label{EQ:secondapprox2}
\end{align}
where  the function $C_1(x)$ is given in~\eqref{EQ:c1}. The refined approximations~\eqref{eq:Eapprox2} take the form 
\begin{align} 
\bar{E}_1(\lambda, k) & = \tilde{E}_1(\lambda, k)-\frac{(1 - \lambda^2)^{2} (\lambda^2 + \beta + 1)}{4} \left(\frac{67}{187} \frac{1}{\lambda^2 \sqrt{\beta (1 + \beta)}} \right. \nonumber \\
& \left. \quad \ + \frac{60}{187} \frac{\sqrt{\beta (1 + \beta)}  -\mathrm{arcsinh}\!\left(\sqrt{\beta}\right)}{\beta^2} \right),  \label{EQ:refinedfirstapprox2} \\
\bar{E}_2(\lambda, k) & = \tilde{E}_2(\lambda, k)-\frac{(1 - \lambda^2)^{3} (\lambda^2 + \beta + 1/2)}{6} \left(\frac{67}{187} \frac{1}{\lambda^2 \sqrt{\beta (1 + \beta)}} \right. \nonumber \\
& \left. \quad \ + \frac{45}{187} \frac{\sqrt{\beta (1 + \beta)}  -\mathrm{arcsinh}\!\left(\sqrt{\beta}\right)}{\beta^2} \right). \label{EQ:refinedsecondapprox2}
\end{align}
Let $\bar{\Delta}_N$ be the number given in~\eqref{EQ:errordiff2}. Then approximation~\eqref{EQ:refinedfirstapprox2} together with~\eqref{EQ:refinedsecondEIremainderbound} puts $E(\lambda, k)$ within an interval of 
length $\bar{\Delta}_1 \cdot E(\lambda, k)$, while~\eqref{EQ:refinedsecondapprox2} puts $E(\lambda, k)$ within an interval of length $\bar{\Delta}_2 \cdot E(\lambda, k)$. 
The results of numerical computation are presented in Table~\ref{TABLE:2}.
\begin{table}[h!]
\small
\begin{tabular}{l l l l l l l l}
 \hline
 $\lambda$ & $k$ & $E(\lambda, k)$ & First order & First order & Relative & Relative & Relative error \\
& & & approx.~\eqref{EQ:firstapprox2} &  approx.~\eqref{EQ:refinedfirstapprox2} & error $\tilde{e}_1$ &  error $\bar{e}_1$ &  range $\bar{\Delta}_1$  \\
 \hline
.8 & .8 & .8501 & .8976 & .8491 & $-.05586$ & $.001162$  &  .08435    \\
.9 & .9 & .9504 & .9532 & .9509 & $-.01343$ & $-.5311 \times 10^{-3}$  &  .01618  \\
.95 & .95 & .9900 & .9933 & .9909 & $-.003344$ & $-.9083 \times 10^{-3}$  &  .003602  \\
.99 & .95 & 1.0572 & 1.0574 & 1.0572 & $-.2771 \times 10^{-3}$ & $-.3481 \times 10^{-4}$  &  $.2784 \times 10^{-3}$   \\ 
.99 & .99 & 1.0056 & 1.0057 & 1.0056 & $-.1355 \times 10^{-3}$ & $-.3648 \times 10^{-4}$  &  $.1335 \times 10^{-3}$   \\
.999 & .99 & 1.0220 & 1.0220 & 1.0220 & $-.4114 \times 10^{-7}$ & $-.5547 \times 10^{-8}$  &  $.4085 \times 10^{-7}$ \\
\hline
\end{tabular}
\begin{tabular}{l l l l l l l l}
$\lambda$ & $k$ & $E(\lambda, k)$ & Second order & Second order& Relative & Relative & Relative error  \\
& & & approx.~\eqref{EQ:secondapprox2} & approx.~\eqref{EQ:refinedsecondapprox2} & error $\tilde{e}_2$ & error $\bar{e}_2$  &  range $\bar{\Delta}_2$ \\
\hline
.8 & .8 & .8501 & .8589 & .8501 & $-.01028$ & $-.2378 \times 10^{-4}$  & $.01771$  \\
.9 & .9 & .9504 & .9516 & .9505 & $-.001286$ & $-.6168 \times 10^{-4}$  &  $.001870$ \\ 
.95 & .95 & .9900 & .9901 & .9900 & $-.1633 \times 10^{-3}$ & $-.1004 \times 10^{-4}$  & $.2188 \times 10^{-3}$  \\
.99 & .95 & 1.0572 & 1.0572 & 1.0572 & $-.3110 \times 10^{-7}$ & $-.8657 \times 10^{-8}$   & $.3212 \times 10^{-7}$  \\ 
.99 & .99 & 1.0056 & 1.0056 & 1.0056 & $-.1345 \times 10^{-7}$ & $-.9252 \times 10^{-9}$ &  $.1689 \times 10^{-7}$  \\ 
.999 & .99 & 1.0220 & 1.0220 & 1.0220 & $-.4774 \times 10^{-10}$ & $-.1502 \times 10^{-10}$  &  $.4679 \times 10^{-10}$   \\ 
\hline
\end{tabular}
\caption{{\small Numerical examples for approximations~\eqref{EQ:firstapprox2},~\eqref{EQ:refinedfirstapprox2},~\eqref{EQ:secondapprox2} and~\eqref{EQ:refinedsecondapprox2} derived from expansion~\eqref{eq:Eapprox2}. The numbers $\tilde{e}_i$ and $\bar{e}_i$ are the  relative errors $(E(\lambda, k) - \tilde{E}_i(\lambda, k))/E(\lambda, k)$ and $( E(\lambda, k) - \bar{E}_i(\lambda, k))/E(\lambda, k)$, respectively, $i = 1, 2$. 
The numbers~$\bar{\Delta}_1, \bar{\Delta}_2$ are given in~\eqref{EQ:errordiff2}.}}
\label{TABLE:2}
\end{table}

We conclude by comparing the above results with the asymptotic approximation~\eqref{EQ:CGexpansion1} from~\cite{Lopez2000, Carlson1994} with the error bounds ~\eqref{EQ:Lbound1} and~\eqref{EQ:CGbound1}, respectively. We denote $\Delta_1^{*}$ to be 
the difference between the upper and the lower bound in~\eqref{EQ:Lbound1} divided by $E(\lambda, k)$,\ie,
\begin{equation} \label{EQ:errordiff3}
\Delta_1^{*} = \text{(rhs of~\eqref{EQ:Lbound1} } -  \text{ lhr of~\eqref{EQ:Lbound1})}/E(\lambda, k).
\end{equation} 
Similarly, Let $\Delta_2^{*}$ be 
the difference between the upper and the lower bound in~\eqref{EQ:CGbound1} divided by $E(\lambda, k)$,\ie,
\begin{equation} \label{EQ:errordiff4}
\Delta_2^{*} = \text{(rhs of~\eqref{EQ:CGbound1} } -  \text{ lhr of~\eqref{EQ:CGbound1})}/E(\lambda, k).
\end{equation} 
The results of numerical computation are given in Table~\ref{TABLE:3}. 
\begin{table}[h!]
\small
\centering
\begin{tabular}{l l l l l l l l l l}
 \hline
 $\lambda$ & $k$ & $E(\lambda, k)$ & First order & Relative  & Relative error & Relative error  \\
& & &  approx.~\eqref{EQ:CGexpansion1} &  error & range $\Delta_1^{*}$ & range $\Delta_2^{*}$   \\
 \hline
.8 & .8 & .8501 & .8343 & $.01864$  & .78055 & .23538 \\
.9 & .9 & .9504 & 1.0127 & $-.06551$  & .66727 & .17444\\
.95 & .95 & .9900 & 1.0704 & $-.08121$  & .44780 & .12025  \\
.99 & .95 & 1.0572 & 1.1178 & $-.05736$  & .27546 & .105 \\ 
.99 & .99 & 1.0056 & 1.0472 & $-.04136$  & .15715 & .03994\\
.999 & .99 & 1.0220 & 1.0434 & $-.02088$  & .07386 & .03581 \\ 
.999 & .999 & 1.0017 & 1.0094 & $-.007712$  & .02327 &  .006137 \\
\hline
\end{tabular}
\caption{{\small Numerical examples for the approximation~\eqref{EQ:CGexpansion1} ($=$\eqref{eq:Lopez1}) due to Carlson-Gustafson and L\'{o}pez. The fifth column equals the relative error $r_1$ in~\eqref{EQ:CGexpansion1} ($=$\eqref{eq:Lopez1}) divided by $E(\lambda, k)$. The numbers $\Delta_1^{*}$ and $\Delta_2^{*}$ are given in~\eqref{EQ:errordiff3} and~\eqref{EQ:errordiff4}, respectively.}}
\label{TABLE:3}
\end{table}

\paragraph{Acknowledgements} We thank the anonymous referees for a number of useful remarks which improve the exposition of the paper. 

\bibliographystyle{abbrv}
\def\cprime{$'$}


\section*{Appendix. Approximations of Carlson-Gustafson and L\'{o}pez}
In this appendix we will convert the asymptotic approximations for symmetric elliptic integrals $R_F$ and $R_D$ due to L\'{o}pez \cite{Lopez2000}
and Carlson-Gustafson \cite{Carlson1994} and their  error bounds into the corresponding results for the Legendre's second incomplete EI $E(\lambda,k)$ defined in \eqref{eq:E-defined}.  As 
$$
E(\lambda, k) = \lambda R_F(1-\lambda^2, 1-k^2\lambda^2, 1) - \frac{1}{3}k^2\lambda^3 R_D(1-\lambda^2, 1-k^2 \lambda^2, 1)  
$$
by \cite[(4.2)]{Carlson1979}, we need the asymptotic approximations for $R_F(a,b,1)$ and $R_D(a,b,1)$ as $a,b\to0$.  In view of the easily verifiable relations
$$
R_F(x,y,z)=z^{-1/2}R_F(x/z,y/z,1),~~~~R_D(x,y,z)=z^{-3/2}R_D(x/z,y/z,1),
$$
it suffices to use the asymptotics of $R_F(x,y,z)$ and $R_D(x,y,z)$ as $z\to\infty$ while $x$ and $y$ remain fixed.  The first approximations from \cite[(3.1)]{Lopez2000} after simple rearrangement are given by  (under the assumption $0\le x<y\le z$)
$$
R_F(x/z,y/z,1)=\ln\left(\frac{2}{\sqrt{x/z}+\sqrt{y/z}}\right)+(\psi(1)-\psi(1/2))/2+R_1^F
$$
with the error bound \cite[(3.5)]{Lopez2000}  (we used the identity $\psi(2)=\psi(1)+1$):
$$
0<R_1^F\le \frac{|x+y|}{8z}\left(\ln\left(1+\frac{2z}{|x+y|}\right)+2\right),
$$
$$
0<R_1^F\le \frac{\max(2,|x+y|)}{\sqrt{z}}.
$$
Further, by \cite[(3.14)]{Lopez2000}  (also under the assumption $0\le x<y\le z$)
$$
R_D(x/z,y/z,1)=3\ln\left(\frac{2}{\sqrt{x/z}+\sqrt{y/z}}\right)+\frac{3}{2}\left(\psi(1)-\psi(3/2)\right)+R_1^D
$$
with
$$
0<R_1^D\le \frac{9|x+y|}{8z}\left(\ln\left(1+\frac{2z}{|x+y|}\right)+2\right),
$$
$$
0<R_1^D\le \frac{2\max(2,|x+y|)}{\sqrt{z}}.
$$
Hence, if we set $1-k^2\lambda^2=x/z$ and $1-\lambda^2=y/z$ we ensure that $0\le x<y$ and we arrive at  (in view of $\psi(3/2)=\psi(1/2)+2$):
$$
E(\lambda, k) = \lambda(1-k^2\lambda^2)\left[\ln\left(\frac{2}{\sqrt{1-k^2\lambda^2}+\sqrt{1-\lambda^2}}\right)
+\frac{1}{2}(\psi(1)-\psi(1/2))\right]+k^2\lambda^3+r_1.
$$
Finally, due to $\sum_{r=1}^{m} \psi(r/m)=m\psi(1)-m\ln(m)$ for $m=2$ we have $\psi(1)-\psi(1/2)=\ln(4)$ and  the approximation takes the form 
\begin{equation}\label{eq:Lopez1}
E(\lambda, k) = \lambda(1-k^2\lambda^2)\ln\left(\frac{4}{\sqrt{1-k^2\lambda^2}+\sqrt{1-\lambda^2}}\right)+k^2\lambda^3+r_1    
\end{equation}
reproduced in \eqref{EQ:CGexpansion1} with the error bounds  
\begin{equation} \label{EQ:Lbound1}
-\frac{3k^2\lambda^3 (2-\lambda^2-k^2\lambda^2)}{8}\left[\ln\frac{4-\lambda^2-k^2\lambda^2}{2-\lambda^2-k^2\lambda^2}+2\right] \le r_1
\le \frac{\lambda(2-\lambda^2-k^2\lambda^2)}{8}\left[\ln\frac{4-\lambda^2-k^2\lambda^2}{2-\lambda^2-k^2\lambda^2}+2\right].
\end{equation}

Combining \cite[(26)]{Carlson1994} and \cite[(34)]{Carlson1994} we again obtain the approximation \eqref{eq:Lopez1}. However, the bounds for remainder term differ from those above and take the form 
\begin{multline} \label{EQ:CGbound1}
\left(\frac{\lambda \sqrt{(1 - \lambda^2)(1 - k^2 \lambda^2)}}{2 \left(1 - \sqrt{(1 - \lambda^2)(1 - k^2 \lambda^2)}\right)} 
- \frac{3 k^2 \lambda^3 (2 - \lambda^2 (1 + k^2))}{2 \lambda^2 (1 + k^2)} \right) \ln \frac{2}{\sqrt{1 - \lambda^2} + \sqrt{1 - k^2 \lambda^2}} < r_1 \\
< \frac{\lambda (2 - \lambda^2 (1 + k^2))}{2 + \lambda^2 (1 + k^2)} \ln \frac{4}{\sqrt{1 - \lambda^2} + \sqrt{1 - k^2 \lambda^2}} - \frac{k^2 \lambda^3 \sqrt{(1 - \lambda^2) (1 - k^2 \lambda^2)}}{1 - \sqrt{(1 - \lambda^2) (1 - k^2 \lambda^2)}} \ln \frac{2}{\sqrt{1 - \lambda^2} + \sqrt{1 - k^2 \lambda^2}}.
\end{multline}
Table~\ref{TABLE:3} shows that these error bounds are more precise than \eqref{EQ:Lbound1} at the price of being substantially more complicated.


%
%
%
%
%
%
%
%


\end{document}